\documentclass[11pt]{amsart}

\usepackage{amsthm}
\usepackage{amsfonts} 
\usepackage{amssymb} 
\usepackage{amsmath} 
\usepackage{mathrsfs}
\newcommand{\leftexp}[2]{{\vphantom{#2}}^{#1}{#2}} \newcommand{\bdm}{\begin{displaymath}} \newcommand{\edm}{\end{displaymath}} \newcommand{\mf}[1]{\mathfrak{#1}} \newcommand{\mb}[1]{\mathbb{#1}} \newcommand{\mc}[1]{\mathcal{#1}} \newcommand{\mr}[1]{\mathrm{#1}} \newcommand{\ms}[1]{\mathsf{#1}} 
\newtheorem{theorem}{Theorem}[section] 
\newtheorem{lemma}[theorem]{Lemma} 
\newtheorem{corollary}[theorem]{Corollary} 
\newtheorem{prop}[theorem]{Proposition} 
\theoremstyle{remark} 
 
\theoremstyle{definition} 
\newtheorem{definition}[theorem]{Definition}

\begin{document}

\title[Depth zero supercuspidal $L$-packets for inner forms of $GSp_4$]{Depth zero supercuspidal $L$-packets\\ for inner forms of $GSp_4$} 
\author{Jaime Lust} 
\address{Department of Mathematics\\
University of California, San Diego\\
La Jolla, 92093}
\curraddr{
Department of Mathematics\\
University of Iowa\\
Iowa City, 52242}
\email{jaime-lust@uiowa.edu}

\begin{abstract}
We show that for any tame regular discrete series parameter of $GSp_4$ or its inner form $GU_2(D),$ the $L$-packet attached by the local Langlands conjecture [GT], [GTan] agrees with the $L$-packet of depth zero supercuspidal representations constructed by DeBacker and Reeder [DR]. 
\end{abstract}

\maketitle

\section{Introduction}

Let $G$ be a linear reductive group over a non-archimedean local field $k$ of characteristic $0$.
The local Langlands conjectures predict that irreducible smooth representations of $G$ should be parametrized by admissible homomorphisms 
\[\phi:W'_k\longrightarrow \leftexp{L}{G},\] of the Weil-Deligne group $W'_k$ to the Langlands dual group $\leftexp{L}{G}$.
The fiber over such a $L$-parameter $\phi$ is a finite set of irreducible smooth representation of $G$ called a $L$-packet. 
Such a map should satisfy certain desired properties that characterize the map uniquely such as the preservation of local factors attached to both sides of the correspondence.

Around 2000, Harris and Taylor [HT], and separately Henniart [He2], proved the local Langlands conjecture for $GL_n$. Cases for small $n$ were established earlier by Kutzko and Henniart. 
Rogawski [Ro] proved the conjecture for $U_2$ and $U_3$.
In 2007, Gan and Takeda [GT] proved the local Langlands conjecture for $GSp_4$. Also, Gan and Tantono [GTan] proved the local Langlands conjecture for $GU_2(D),$ the inner form of $GSp_4.$

For supercuspidal representations, there is another conjectural classification which is independent of the local Langlands conjectures. 
It is conjectured that any irreducible supercuspidal representation $\pi$ of $G$ is of the form 
\[\pi\cong \mr{c-Ind}_{K}^{G}\sigma,\] where $K$ is an open compact mod center subgroup of $G$ and $\sigma$ is a representation of $K.$
Recently much work has been done in this direction. In 1993 Bushnell and Kutzko [BK1] showed that any irreducible supercuspidal representation of $GL_n$ is compactly induced.
In 2001, J.K. Yu [Y] constructed a family $(K,\sigma)$ for any connected reductive group $G$.
In 2007, J.-L. Kim [Ki] proved that, for $p$ large, the family constructed by Yu exhausts all supercuspidal representations of $G$. Independently, for $p\ne 2$, 
Stevens [St] constructed another family $(K,\sigma)$ for the classical groups $U_n, Sp_{2n}$ or  $SO_n$. He also proved that for these groups $G$, his family exhausts all supercuspidal representations of $G$.

It is not obvious how to relate the construction of supercuspidal representations via compact induction to the classification given by the local Langlands correspondence. There is a series of papers by Bushnell and Henniart [BH1], [BH2] devoted to answering this question for $GL_n$. 
In the tame case for $GL_n$, these results were established earlier by Henniart [He1]. 
The purpose of this paper is towards answering this question for $GSp_4$. 
 
\sloppy 

For pure inner forms of unramified $p$-adic groups, DeBacker and Reeder [DR] give a 
parametrization 
of tame regular discrete series Langlands parameters. 
For any such parameter they attach a $L$-packet of compactly induced depth zero supercuspidal representations. 
For the groups $SO_{2n+1}$ and $Sp_{2n}$, Savin [Sa] considered the case of the generic depth zero supercuspidal representation attached to a tame regular discrete series parameter. He showed that the parametrization given by DeBacker and Reeder agrees with the lifting of generic supercuspidal representations to a general linear group as given by Jiang and Soudry [JS] in the case of $SO_{2n+1}$, and by Cogdell, Kim, Piatetski-Shapiro, and Shahidi [CKPS] in the case of $Sp_{2n}.$
We follow the general strategy for proof as in [Sa], however we consider the entire $L$-packet attached to a parameter, which also covers the case of non-generic representations.

\fussy

We consider the subset of tame regular discrete series parameters of $GSp_4$. 
The construction of DeBacker and Reeder applies to $GSp_4$ but not to $GU_2(D).$ We extend their construction to give $L$-packets of depth zero supercuspidal representations for both $GSp_4$ and $GU_2(D)$, that agree with DeBacker and Reeder for $GSp_4.$
We note that recently Kaletha [Ka], by an alternate construction based on work of Kottwitz on isocrystals, has extended the work of DeBacker and Reeder to non-pure inner forms of an unramified group $G.$ 

Our main result is that the $L$-packets given by DeBacker and Reeder agree with those given by the local Langlands conjectures for $GSp_4$ and $GU_2(D):$ 

\begin{theorem} Let $\phi$ be a tame regular discrete series $L$-parameter.
Let $L_\phi^{DR}$ be the $L$-packet of depth zero supercuspidal representations of $GSp_4(k)$ or $GU_2(D)$ corresponding to $\phi$ by the construction of DeBacker and Reeder given in Section 4.5. Let $L_\phi^{GT}$ be the $L$-packet of supercuspidal representations of $GSp_4(k)$ or $GU_2(D)$ corresponding to $\phi$ via the local Langlands conjecture for $GSp_4$ or $GU_2(D).$ Then
\bdm L_\phi^{DR}=L_\phi^{GT}.\edm  \end{theorem}

In [GT] and [GTan], the local Langlands classification is characterized by the preservation of $L$-factors and $\epsilon$-factors. 
To prove Theorem 1.1, we need to show that certain $L$-functions have poles at $s=0$. By theory of Shahidi [Sh], this question is equivalent to studying the reducibility points of the generalized principal series
\bdm I(s,\pi\boxtimes\sigma)=\mr{Ind}_{P}^{G} \delta_{P}^{1/2} \pi\boxtimes\sigma\vert \mr{det} \vert^{s}.\edm
Here, $\pi$ is an irreducible representation of $GSpin_{5}(k)\cong GSp_4(k)$ or $GSpin_{4,1}(k)\cong GU_2(D)$, and $\sigma$ a representation of $GL_{2m}(k)$, where
$G=GSpin_{4m+5}(k)$ or $GSpin_{2m+4, 2m+1}(k),$ and $P=M\cdot N$ is the parabolic with Levi factor
$M=GSpin_{5}(k)\times GL_{2m}(k)$ or $GSpin_{4,1}(k)\times GL_{2m}(k).$

We need to show:

\begin{theorem} Let $\pi$ be a depth zero supercuspidal representation of $GSpin_{5}(k)$ or $GSpin_{4,1}(k)$ corresponding to a tame regular discrete series $L$-parameter $\phi=\phi_{1}\oplus\dots\oplus\phi_{r}$, $r=1,2,$ by the construction of DeBacker and Reeder given in Section 4.5.
Let $\sigma\cong\sigma_{\phi_i}$, where $1\le i\le r$, be the depth zero supercuspidal representation of $GL_{2m}(k)$ attached to the $L$-parameter $\phi_{i}$ via the local Langlands correspondence for $GL_{2m}.$
Then the generalized principal series $I(s,\pi\boxtimes\sigma)$ reduces at a unique $s_{0}>0.$ \end{theorem}

The proof of Theorem 1.2 is achieved by studying a Hecke algebra $\mc{H}(G, \rho)$ associated to this family of induced representations. 
Using the theory of
types and covers developed by Bushnell and Kutzko [BK2], we find a parahoric subgroup $\mc{P}$ of $G$ and representation $\rho$ of $\mc{P}$ such that
irreducible representations in the Bernstein component of $I(s,\pi\boxtimes\sigma)$ are parametrized by simple $\mc{H}(G,\rho)$-modules.  
We use the method of Kutzko and Morris [KM] to give a presentation for $\mc{H}(G,\rho)$ and compute its parameters. 
As $c-\mr{Ind}_\mc{P}^G \rho$ is a depth zero representation, using Morris [Mo2] we describe explicit generators and relations for the Hecke algebra \[\mc{H}(G,\rho)\cong\mr{End}_G(c-\mr{Ind}_\mc{P}^G \rho).\]
We have
$\mc{H}(G,\rho)=\langle T_0,T_1,T_2\rangle$, subject to the relations
\[T_0\,T_i=T_i\,T_0,\quad  T_{i}^{2}=(p_i-1)T_i+p_i, \quad i=1,2.\]
We see $\mc{H}(G,\rho)$ is a Hecke algebra of type $\tilde{A}_1$ tensored with a polynomial algebra.
For $T_i, \, i=1,2$, we show there is a subalgebra of $\mc{H}(G,\rho)$ which is canonically isomorphic to the endomorphism ring of an induced representation over a finite field, with $T_i$ identified as the unique non-identity generator. 
In this way computation of the parameters $p_i$ reduce to computations over a finite field, which can be done following a theorem of Lusztig [Lu].
We show the parameters $p_i, i=1,2$, of the Hecke algebra $\mc{H}(G,\rho)$ are unequal.
Theorem 1.2 then follows using results of Matsumoto [Ma] on the Plancherel measure for Hecke algebras of type $\tilde{A}_1$ and work of Harish-Chandra [Sil] on points of reducibility of principal series.

The organization of this paper is as follows.
We begin by giving a short description of general symplectic groups and general spin groups in Section 2. In Section 3, we give a complete characterization of tame regular discrete series Langlands parameters for $GSp_{2n}.$ In Section 4 we define the $L$-packets $L_\phi^{GT}$ and $L_\phi^{DR}$ for $GSp_4$ and $GU_2(D)$ attached to a tame regular discrete series parameter $\phi.$ 
Properties of the local Langlands conjecture for $GSp_4$ and $GU_2(D)$, proved in [GT] and [GTan] respectively, are described in Section 4.1 and Section 8. In the situation where a theory of $L$-factors and $\epsilon$-factors has not been fully developed, such as in the case of non-generic supercuspidal representations for $GSp_4,$ the local Langlands correspondence is characterized by the preservation the coarser invariant Plancherel measure. 
In Section 4.2 we briefly review the DeBacker-Reeder construction.
As noted above, this construction only applies to pure inner forms of an unramified $p$-adic group $G,$ so in Section 4.3 we extend their construction to inner forms in the case that after taking $K$-points, where $K$ is the maximal unramified extension of $k$, the adjoint quotient \[j:G(K)\rightarrow G_{ad}(K)\] remains surjective.

In Section 5, we introduce the generalized principal series $I(s,\pi\boxtimes\sigma)$ and describe its Bernstein component. In Sections 6 we give a presentation of the Hecke algebra $\mc{H}(G,\rho)$ and compute its parameters $p_i$ using results of Morris and Lusztig. 
We note that [Lu, 8.6], which explicitly describes the endomorphism algebra of the induced space $\mr{Ind}_{\ms{M}\ms{N}}^{\ms{G}}(\rho)$ where $\ms{G}$ is a finite group of Lie type and $\rho$ is an irreducible cuspidal representation of $\ms{M}$, is subject to the condition that $\ms{G}$ has connected center. To apply the theorem, in all cases we embed our groups $\ms{G}$ into groups $\ms{G}'$ that have connected center.
Then in Sections 7 and 8, we state the main theorem and show that the $L$-packets $L_\phi^{GT}$ and $L_\phi^{DR}$ agree.
Throughout the paper we identify certain connected reductive linear algebraic groups in terms of their root datum. The appendix gives a description of the root datum of these groups.
\vskip 10pt

\noindent
\textbf{Acknowledgements:} This paper is based on the author's Ph.D. thesis. The author is grateful to her advisor Wee Teck Gan for suggesting this problem and for his encouragement and advice. The author also thanks Philip Kutzko, Lawrence Morris, Gordan Savin, and Jiu-Kang Yu for helpful conversations.

\section{General symplectic and general spin groups}

Let $k$ be a non-archimedian local field of characteristic 0 with finite residue field $\mf{f}$ of characteristic $p.$ Let $F=k$ or $\mf{f}$.



\subsection{General symplectic groups}

Let $V_1=Ff_1\oplus Ff_2$ be the $2$ dimensional vector space over $F$ equipped with the non-degenerate alternating bilinear form $\langle\,\,\,,\,\,\rangle$ given by 
\[\langle f_1,f_2\rangle=-\langle f_2,f_1\rangle=1,\quad \langle f_i,f_i\rangle=0.\]
Then $V=V_1^{\oplus n}$ is a symplectic space of dimension $2n$ over $F$.
Let
\bdm GSp_{2n}:=GSp(V)=\{g\in GL(V): \,\langle gv_1, gv_2\rangle=\lambda(g)\langle v_1, v_2\rangle\}\edm
where $\lambda(g)\in F^\times$ is a scalar.
The scalar $\lambda(g)$ is multiplicative and
$\mr{sim}:GSp(V)\longrightarrow F^\times$ where $\mr{sim}(g)=\lambda(g)$ is the similitude character of $GSp(V).$
We have \[Sp(V)=\{g\in GSp(V):\, \lambda(g)=1\}.\]

We now define $GU_{2n}(D)$, an inner form of $GSp_{4n}.$
Let $D$ be the quaternion division algebra over $k.$
Let $V_2=De_1\oplus De_2$ be the 2-dimensional vector space over $D$ equipped with a Hermitian form with inner product given by
 \[ \langle e_1, e_2\rangle=\langle e_2, e_1\rangle=1,\quad \langle e_i, e_i\rangle=0.\]
Also, \bdm \langle dv_1,d'v_2\rangle=\tau(d)\langle v_1,v_2\rangle d'\edm where $\tau$ is the standard involution on $D.$
Then $V=V_2^{\oplus n}$ is a Hermitian space of dimension $2n$ over $D$ and 
\bdm GU_{2n}(D)=\{g\in \mr{Aut}_D(V):\,\langle gv_1, gv_2\rangle=\mu(g)\langle v_1, v_2\rangle\}\edm
where $\mu(g)\in k^\times$ is a scalar.

\subsection{General spin groups}

In this section we define general spin groups $GSpin_m$, and related groups.

Let $V$ be an $m$ dimensional vector space over $F$ equipped with a non-degenerate quadratic form $\langle\,\,\,,\,\,\rangle.$
Then \bdm GO(V)=\{g\in GL(V): \,\langle gv_1, gv_2\rangle=\lambda(g)\langle v_1, v_2\rangle\}\edm 
where $\lambda\in F^\times$ is a scalar.
The similitude character of $GO(V)$ is $\mr{sim}: GO(V)\longrightarrow F^\times$ where  $\mr{sim}(g)=\lambda(g).$
Let $GSO(V)$ be the connected component of $GO(V)$, as the latter is not a connected algebraic group.
We have \bdm SO(V)=\{g\in GSO(V): \, \lambda(g)=1\}.\edm
Let $C(V)=C^+(V)\oplus C^-(V)$ be the Clifford algebra of $V$ with its 2-grading. There is a canonical embedding $V\hookrightarrow C^-(V).$
Then
\bdm GSpin(V)=\{g\in C^+(V)^\times\vert\,\, gg^\iota=\nu(g),\textrm{ and } gVg^{-1}=V\}\edm
where $\iota$ is the main involution of $C(V)$ and $\nu(g)$ is a scalar.
The scalar $\nu(g)$ is multiplicative and
$\mr{sim}:GSpin(V)\longrightarrow F^\times$ where $\mr{sim}(g)=\nu(g)$ is the similitude character of $GSpin(V).$
Denote by \bdm Spin(V)=\{g\in GSpin(V):\,\nu(g)=1\}.\edm
As algebraic groups, we have the exact sequence
\bdm 1\longrightarrow Z^0 \longrightarrow GSpin(V)\longrightarrow SO(V)\longrightarrow 1,\edm
where $Z^0$ is the connected center of $GSpin(V).$

Over a $p$-adic field $k$, there are two quadratic forms of dimension $m$ and discriminant 1.
Let $V^+=H^{2n}\oplus\langle 1\rangle$ or $V^+=H^{2n}$ be the split quadratic space of dimension $m=2n+1$ or $2n$ and discriminant 1.
Here $H=ke_1\oplus ke_2$ is the hyperbolic plane with inner form given by
\[\langle e_1, e_2\rangle=\langle e_2, e_1\rangle=1,\quad \langle e_i, e_i\rangle=0.\]
Let $V^-=H^{2(n-1)}\oplus D_0 $ be the non-split quadratic space of dimension $2n+1$ and discriminant 1,
where $D_0$ is the subset of the quaternion $k$-algebra $D$ of elements of reduced trace 0 and the quadratic form on $D_0$ is given by the reduced norm.
Let \bdm GSO_m=GSO(V^+),\quad SO_m=SO(V^+).\edm
Define \bdm GSpin_m=GSpin(V^+), \edm 
and
\bdm GSpin_{n+2,n-1}=GSpin(V^-). \edm

Over a finite field $\mf{f}$ of characteristic $p$, let
$V^+=H^{2n}\oplus\langle 1\rangle$ or $V^+=H^{2n}$ be the quadratic space of dimension $m=2n+1$ or $2n$ over $\mf{f}$ where $H=\mf{f}e_1\oplus\mf{f}e_2$ is the hyperbolic plane with inner product given as above.
Let \bdm GSpin_m=GSpin(V^+).\edm
For $\mf{f}=\mb{F}_q$,
let $V^-=H^{2(n-1)}\oplus \mb{F}_{q^2}$ where $\mb{F}_{q^2}$ is the unique quadratic extension of $\mb{F}_q$ with quadratic form given by $N_{\mb{F}_{q^2}/\mb{F}_q}.$
Let \bdm \leftexp{2}{GSpin}_{2n}=GSpin(V^-).\edm

\section{Parameters}

We first set some notation. Let $k$ be a non-archimedean local field of characteristic 0 with finite residue field $\mf{f}$ of characteristic $p$. Let $q=\vert\mf{f}\vert.$ Let $\bar{k}$ be a fixed algebraic closure of $k$, $K$ the maximal unramified extension of $k$ in $\bar{k}$, and $k_m$ the unramified extension of degree $m$ of $k$ in $\bar{k}.$ Let $\bar{\mf{f}}$ be the residue field of $K$.  Then $\bar{\mf{f}}$ is an algebraic closure of $\mf{f}.$ 

Let $\mc{I}$ be the inertia subgroup of the Galois group $\mr{Gal}(\bar{k}/k)$. We have 
$\mr{Gal}(\bar{k}/k)/\mc{I}\cong \mr{Gal}(K/k)\cong \mr{Gal}(\bar{\mf{f}}/\mf{f}).$ 
A geometric Frobenius $\mr{Frob}$ is one whose image in $\mr{Gal}(\bar{\mf{f}}/\mf{f})$ is the automorphism which acts as the inverse of $x\mapsto x^q$ on $\bar{\mf{f}}.$
Fix a choice of geometric Frobenius.  Then $\langle\mr{Im}(\mr{Frob})\rangle\subset \mr{Gal}(\bar{\mf{f}}/\mf{f})$ is a dense subgroup and the Weil group is
\[W_k=\mc{I}\rtimes\langle\mr{Frob}\rangle.\]
The Weil-Deligne group $W'_k$ is defined as 
$W'_k=W_k\times SL_2(\mb{C}).$

Let $G$ be a connected reductive linear algebraic group and $T$ a maximal torus of $G.$ Let $X=X^*(T)$ and  $X^\vee=X_*(T)$ be the groups of algebraic characters and cocharacters of $T$. Let $\Phi$ and  $\Phi^\vee$ be the sets of roots and coroots of $T$. 
The quadruple \bdm \Psi=(X,\Phi,X^\vee,\Phi^\vee)\edm is the root datum for $G$.
Up to isomorphism, there is a unique complex reductive group $\hat{G}$ with root datum $(X^\vee, \Phi^\vee, X, \Phi)$ dual to that of $G.$  We have $\hat{G}$ is the Langlands dual group of $G$. 

The dual group of $GSp_4$ is $GSp_4(\mb{C}).$
By a Langlands parameter for $GSp_4$ we mean a continuous homomorphism
 \bdm \phi: W'_k \longrightarrow GSp_4(\mb{C}) \edm such that $\phi(\mr{Frob})$ is semisimple. 
Such a parameter is said to be admissible, and homomorphisms are taken up to $\hat{G}$ conjugacy.

\begin{definition} Let $\phi$ be a Langlands parameter.
\begin{itemize}
\item[(i)] We say that $\phi$ is tame if $\phi$ is trivial on the wild inertia group $\mathcal{I}^{+},$ the maximal pro-$p$ subgroup of $\mc{I}$.
\item[(ii)] We say the $\phi$ is regular if the centralizer in $\hat{G}$ of $\phi (\mathcal{I})$ is a maximal torus $\hat{T}$ in $\hat{G}$.
\item[(iii)] We say that $\phi$ is discrete series if the identity component of the centralizer in $\hat{G}$ of $\phi(W_k)$ is equal to the identity component of the center $\hat{Z}$ of $\hat{G}.$
\end{itemize}
\end{definition}

For $k_{2d}$ the unramified extension of $k$ of degree $2d$, the Weil group of $k_{2d}$ is \bdm W_{k_{2d}}=\langle \mr{Frob}^{2d}\rangle\ltimes\mathcal{I}.\edm

\begin{lemma} Fix $\lambda: W_k\longrightarrow \mb{C}^{\times}.$ Let $V(\eta)=\mr{Ind}_{W_{k_{2d}}}^{W_k}\eta,$ where $\eta: W_{k_{2d}}\longrightarrow\mb{C}^\times$ is a continuous character. Then
\begin{itemize}
\item[(i)] $V(\eta)$ is tame $\Longleftrightarrow$ $\eta$ is trivial on the wild inertia subgroup $\mc{I}^{+}$,
\item[(ii)] $V(\eta)$ is irreducible $\Longleftrightarrow$ $\eta, \eta^{\mr{Frob}}, \dots, \eta^{\mr{Frob}^{2d-1}}$ are pairwise distinct,
\item[(iii)] Under $\left(ii\right),$ $V(\eta)$ is symplectic with similitude character $\lambda$ $\Longleftrightarrow$ $\eta\cdot\eta^{\mr{Frob}^{d}}\cong\lambda\vert_{W_{k_{2d}}}$ and\\ 
$\eta(\mr{Frob}^{2d})=-\lambda(\mr{Frob}^{d})$.
\end{itemize}
\end{lemma}

\begin{proof}  (i) Since $\mc{I}\subset W_{k_{2d}},$ by the definition of induced representation
$\phi: W_k\longrightarrow GL(V(\eta))$ is trivial on $\mathcal{I}^{+}$ if and only if $\eta$ is trivial on $\mathcal{I}^{+}$. 

(ii) A set of coset representatives for $W_{k_{2d}}\setminus W_k\,/\,W_{k_{2d}}$ is $\{\mr{Frob}, \mr{Frob}^{2},\dots, \mr{Frob}^{2d}\}.$ 
Define $\eta^{\mr{Frob}^{i}}: W_{k_{2d}}\longrightarrow\mb{C}^{\times}$ by \bdm \eta^{\mr{Frob}^{i}}(w)=\eta((\mr{Frob}^{i})^{-1}w\mr{Frob}^{i}).\edm  
By Mackey's irreducibility criterion, $V(\eta)$ is irreducible if and only if 
\begin{itemize}
\item[(a)] $\eta$ is irreducible, and 
\item[(b)] the conjugates of $\eta$, namely $\eta, \eta^{\mr{Frob}}, \dots, \eta^{\mr{Frob}^{2d-1}},$ are pairwise distinct. 
\end{itemize}
Therefore, since $\eta$ is 1-dimensional, $V(\eta)$ is irreducible if and only (b) is satisfied.

(iii) Assume (ii) is satisfied. 
Let $\lambda$ be a character of $W_k.$ There exists a $W_k$-equivariant map 
\bdm B:V(\eta)\otimes V(\eta)\longrightarrow\lambda\edm if and only if $V(\eta)\cong V(\eta)^{\vee}\otimes\lambda.$
Since $V(\eta)$ is irreducible, any such nonzero map $B$ would be nondegenerate.
Now, \bdm V(\eta)^{\vee}\otimes\lambda\cong V(\eta^{\vee})\otimes\lambda\cong V(\eta^{-1})\otimes\lambda\cong V(\eta^{-1}\cdot\lambda\vert_{W_{k_{2d}}}),\edm
 where $\eta^{\vee}=\eta^{-1}.$ By Frobenius reciprocity, 
 \bdm \mr{Hom}_{W_k}(V(\eta^{-1}\cdot\lambda\vert_{W_{k_{2d}}}),V(\eta))=\mr{Hom}_{W_{k_{2d}}}(\eta^{-1}\cdot\lambda\vert_{W_{k_{2d}}},V(\eta)\vert_{W_{k_{2d}}}).\edm
 Also by Frobenius reciprocity 
 \bdm V(\eta)\vert_{W_{k_{2d}}}\cong \eta\oplus\eta^{\mr{Frob}}\oplus\dots\oplus\eta^{\mr{Frob}^{2d-1}}.\edm
 Therefore, \bdm V(\eta)\cong V(\eta)^\vee\otimes\lambda \Longleftrightarrow  \eta^{-1}\cdot \lambda\vert_{W_{k_{2d}}}\cong\eta^{\mr{Frob}^i} \Longleftrightarrow \lambda\vert_{W_{k_{2d}}}\cong\eta\cdot\eta^{\mr{Frob}^i} \edm
for some $0\le i\le 2d-1.$

 Assume such a $W_k$-equivariant map $B$ exists. Then $\lambda\vert_{W_{k_{2d}}}\cong \eta\cdot\eta^{\mr{Frob}^i}$ for some $0\le i\le 2d-1.$ 
 Since $\lambda$ is a character of $W_k$, for any power $l$,
 \bdm (\lambda\vert_{W_{k_{2d}}})^{\mr{Frob}^{l}}(w)=\lambda(\mr{Frob}^{-l}w\mr{Frob}^l)=\lambda(\mr{Frob}^{-l})\lambda(w)\lambda(\mr{Frob}^l)=\lambda\vert_{W_{k_{2d}}}(w).\edm
Then \[  \eta\cdot\eta^{\mr{Frob}^i}=(\eta\cdot\eta^{\mr{Frob}^i})^{\mr{Frob}^{l}}=\eta^{\mr{Frob}^{l}}\cdot\eta^{\mr{Frob}^{i+l}}\] for any power $l.$ Taking $l=i$, it follows $\eta^{\mr{Frob}^{2i}}=\eta.$
 Since $0\le i\le 2d-1,$ this is only possible if $i=0$ or $d$ since $\eta, \eta^{\mr{Frob}}, \dots, \eta^{\mr{Frob}^{2d-1}}$ are pairwise distinct. Also, since $\eta, \eta^{\mr{Frob}}, \dots, \eta^{\mr{Frob}^{2d-1}}$ are pairwise distinct, \[\eta\cdot\eta\ne \eta\cdot\eta^{\mr{Frob}^{d}}\] so that $\lambda\vert_{W_{k_{2d}}}\cong \eta\cdot\eta^{\mr{Frob}^i}$ for $i=0$ or $i=d$ but not both.

  
We need to determine under what conditions $B$ is symplectic. By Frobenius reciprocity \bdm V(\eta)\vert_{W_{k_{2d}}}=\bigoplus_{i=0}^{2d-1}\eta^{\mr{Frob}^{i}}.\edm Then 
\bdm B:\Bigg(\bigoplus_{i=0}^{2d-1}\eta^{\mr{Frob}^{i}}\Bigg)\otimes\Bigg(\bigoplus_{i=0}^{2d-1}\eta^{\mr{Frob}^{i}}\Bigg)\longrightarrow\lambda\vert_{W_{k_{2d}}}\edm
is a $W_{k_{2d}}$-equivariant map. 
With respect to this basis, by restriction, \[B:\eta^{\mr{Frob}^{i}}\otimes\eta^{\mr{Frob}^{j}}\longrightarrow\lambda\vert_{W_{k_{2d}}}\] is a nonzero 
$W_{k_{2d}}$-invariant map if the entry $t_{ij}$ in the matrix of the form $B$ is nonzero.
In the case $\lambda\vert_{W_{k_{2d}}}\cong \eta\cdot\eta,$
\bdm t_{ij}\ne 0 \Longleftrightarrow \eta^{\mr{Frob}^{i}}\cdot\eta^{\mr{Frob}^{j}}\cong\lambda\vert_{W_{k_{2d}}} \Longleftrightarrow i=j.\edm
We see in this case $B$ is not symplectic. 
In the case $\lambda\vert_{W_{k_{2d}}}\cong \eta\cdot\eta^{\mr{Frob}^d},$
\bdm t_{ij}\ne 0 \Longleftrightarrow \eta^{\mr{Frob}^{i}}\cdot\eta^{\mr{Frob}^{j}}\cong\lambda\vert_{W_{k_{2d}}} \Longleftrightarrow \vert j-i \vert= d.\edm
We can choose eigenvectors in the subspaces $\eta$ and $\eta^{\mr{Frob}^d}$ of $V(\eta)\vert_{W_{k_{2d}}}$ 
such that
the matrix of $\phi(\mr{Frob}^{d})$ on $\eta\oplus\eta^{\mr{Frob}^{d}}$ has the form $\begin{pmatrix}0 & 1 \\ \eta(\mr{Frob}^{2d}) & 0\end{pmatrix}.$ This matrix preserves $\begin{pmatrix}0 & t_{0d}\\ t_{d0} & 0\end{pmatrix}$ up to $\lambda$, where as $B$ is either symplectic or orthogonal, either $t_{0d}=-t_{d0}$ or $t_{0d}=t_{d0}.$ 
Therefore $\begin{pmatrix}0 & t_{0d}\\ t_{d0} & 0\end{pmatrix}$ is preserved up to $\lambda$ if and only if \bdm \eta(\mr{Frob}^{2d})=\pm\lambda(\mr{Frob}^{d}).\edm Hence, $B$ is symplectic if and only if $\eta(\mr{Frob}^{2d})=-\lambda(\mr{Frob}^{d}).$ 

 Assume $\eta\cdot\eta^{\mr{Frob}^{d}}\cong\lambda\vert_{W_{k_{2d}}}$ and $\eta(\mr{Frob}^{2d})=-\lambda(\mr{Frob}^{d}).$
Extend $\eta\cdot\eta^{\mr{Frob}^{d}}$ to a character of $W_k$ such that the extension is isomorphic to $\lambda$. Then there exists a nondegenerate $W_k$-equivariant map $B:V(\eta)\otimes V(\eta)\longrightarrow\lambda.$ By what we have said above $B$ is symplectic.
\end{proof}

\begin{lemma} All tame regular discrete series Langlands parameters for $GSp_{2n}(k)$ with similitude character $\lambda$ are of the form $\phi: W_k\longrightarrow GSp(V)$ such that 
\bdm V=V(\eta_{1})\oplus\dots\oplus V(\eta_{s}) \hspace{.2cm}where\hspace{.2cm} V(\eta_{i})=\mathrm{Ind}_{W_{k_{2d_i}}}^{W_k}\eta_{i}\edm 
satisfying the conditions of Lemma 3.2, where $k_{2d_i}$ are unramified extensions of $k$ of degree $2d_{i}$ such that $d_{1}+\dots+d_{s}=n.$ In addition,
\begin{itemize}
\item[\textup{(}i\textup{)}] if $k_{2d_i}=k_{2d_j},$ $\eta_{i}$ is not equal to any conjugate $\eta_{j}^{\mr{Frob}^{k}}, 0\leq k\leq d_{j},$ of $\eta_{j}$ so the $V(\eta_{i})$ are pairwise non-isomorphic and 
\item[\textup{(}ii\textup{)}] the similitude character of $V(\eta_{i})$ is $\lambda$ for all $i.$
\end{itemize}
\end{lemma}

\begin{proof} Let $\phi: W_k\rightarrow GSp(V),$ where $V$ is a $2n$ dimensional complex vector space, be a tame regular discrete series Langlands parameter with similitude character $\lambda.$
As $\phi$ is discrete series, the identity component of the centralizer in $GSp(V)$ of $\phi(W_k)$ is equal to the identity component of the center of $GSp(V)$, which implies that $V$ is multiplicity free. We have, \bdm V=\bigoplus V_i\edm where any two $V_i$ are pairwise non-isomorphic and each $V_i$ is symplectic with similitude character $\lambda$. Therefore, it suffices to show any irreducible tame regular $\phi: W_k\rightarrow GSp(V)$ with similitude character $\lambda$ is of the form $V(\eta)$ with $\eta$ as in Lemma 3.2. 

Let $\phi$ be such a parameter where dim $V=2d$.
As $\phi$ is tame, it factors through the tame inertia group $\mathcal{I}_{t}=\mathcal{I}/\mathcal{I}^{+}\simeq \varprojlim \mathfrak{f}_{n}^{\times},$ where $\mathfrak{f}_{n}$ is the degree $n$ extension of $\mathfrak{f}$ the residue field of $k.$
Then, $\phi$ factors through $\mathfrak{f}_{n}^{\times}$ for some $n\geq 1$ and since $\mathfrak{f}_{n}^{\times}$ is cyclic, $\phi(\mc{I}_t)$ is cyclic. There is a basis of $V$ such that
\bdm \phi\vert_{\mc{I}_t}=\bigoplus_i \chi_i, \edm where the $\chi_i$ are characters.
Since $\phi$ is regular the centralizer of $\phi(\mc{I})$ is a maximal torus, so the $\chi_i$ are pairwise distinct.
Since $V$ is irreducible, $\mr{Frob}$ permutes the set $\{\chi_i\}$ transitively. 
Note that $\mr{Frob}^{2d}$ induces the trivial permutation on $\{\chi_i\}.$
Choose $\eta\in\{\chi_i\}.$ Then 
\bdm \mr{Stab}_{W_k}\eta=\langle \mr{Frob}^{2d}\rangle\ltimes\mc{I}=W_{k_{2d}}.\edm
By Frobenius reciprocity \bdm \mr{dim}(\mr{Hom}_{W_{k_{2d}}}(\eta, \phi\vert_{W_{k_{2d}}}))=\mr{dim}(\mr{Hom}_{W_k}(\mr{Ind}_{W_{k_{2d}}}^{W_k} \eta, \phi)).\edm 
Therefore \bdm \phi\cong \mathrm{Ind}_{W_{k_{2d}}}^{W_k}\eta=:V(\eta).\edm 
We have that $V(\eta)$ is tame, irreducible, and symplectic with similitude character $\lambda,$ and therefore satisfies the conditions of Lemma 3.2. 
\end{proof}

\section{Packets}

\subsection{Local Langlands for $GSp_4$ and $GU_2(D)$}

In [GT], Gan and Takeda prove the local Langlands conjecture for $GSp_4.$ They define a surjective finite-one-map 
\bdm L:\Pi(GSp_4)\longrightarrow\Phi(GSp_4)\edm 
from the set of isomorphism classes of irreducible smooth representations of $GSp_4(k)$ to the set of equivalence classes of admissible homomorphisms 
$\phi:W'_k \longrightarrow GSp_4(\mb{C})$ taken up to $GSp_4(\mb{C})$ conjugacy.  
Let $L_\phi^{GT}$ be the fiber of $L$ over $\phi.$ 
The map $L$ [GT, \S 1 Main Theorem] satisfies many expected and desired properties, such as preservation of local factors attached to both sides of the correspondence, that determine the map uniquely. We will discuss some of these properties further in section 8. A packet $L_\phi^{GT}$ is parametrized by the set of irreducible characters of the component group \[A_\phi=\pi_0(C_{GSp_4(\mb{C})}(Im(\phi))).\]

In [GTan], Gan and Tantono extend the local Langlands correspondence for $GSp_4$ to an analogous result for the inner form $GU_2(D).$  
They define a surjective finite-to-one map
\bdm L:\Pi(GU_2(D))\longrightarrow \Phi(GU_2(D))\edm
from the set of isomorphism classes of irreducible smooth representations of $GU_2(D)$ to the set of relevant $L$-parameters for $GU_2(D).$
A $L$-parameter $\phi$ in $\Phi(GSp_4)$ is said to be relevant for $GU_2(D)$ if it does not factor through any irrelevant parabolic subgroup. Also let $L_\phi^{GT}$, the $L$-packet of representations of $GU_2(D)$, denote the fiber of $L$ over a relevant parameter $\phi.$ The map $L$ satisfies analogous properties to those for $GSp_4$ which characterize the map uniquely.

We have the modified component group
\bdm B_\phi=\pi_0(C_{Sp_4(\mb{C})}(Im(\phi))).\edm 
We have an injection of the group of irreducible characters 
$\mr{Irr}(A_\phi) \hookrightarrow \mr{Irr}(B_\phi)$
which identifies $\mr{Irr}(A_\phi)$ as the subgroup of characters of $B_\phi$ which are trivial on the image of the center $Z_{Sp_4(\mb{C})}$ of $Sp_4(\mb{C}).$
A parameter $\phi$ is relevant for  $GU_2(D)$ if and only if  $\mr{Irr}(B_\phi)\ne\mr{Irr}(A_\phi)$ and an $L$-packet $L_\phi^{GT}$ for $GU_2(D)$
is naturally parametrized by the set $\mr{Irr}(B_\phi)\setminus\mr{Irr}(A_\phi).$

\subsection{DeBacker-Reeder $L$-packets}

Given a TRD parameter $\phi$ of an unramified $p$-adic group $G$, by explicit construction, DeBacker and Reeder associate an $L$-packet of depth zero supercuspidal representations distributed among the pure inner forms of $G.$ 
In the following, for $GSp_4\cong GSpin_5$ we extend their construction to associate to a TRD parameter $\phi$ an $L$-packet $\Pi(\phi)$ of depth zero supercuspidal representations distributed among the inner forms of $GSp_4.$

The inner forms of a $p$-adic group $G$ are parametrized by classes in the Galois cohomology set $H^{1}(k,G_{ad}).$ 
We have that $\mr{H}^1(k, PGSp_4)=\{\pm 1\}$ and the group $GSp_4$ has one inner form, namely $GU_2(D).$ 
Members of the $L$-packet $\Pi(\phi)$ corresponding to a parameter $\phi$ are parametrized by the irreducible characters of the component group $\mr{Irr}(B_\phi),$ where 
\bdm B_\phi=\pi_0(C_{\widehat{G_{ad}}}(Im(\phi))).\edm 
By restriction, any $\rho \in \mathrm{Irr}(B_{\phi})$ determines a character on $\pi_{0}(Z(\leftexp{L}{G_{ad}})).$ Then via Kottwitz' isomorphism [Ko] which says $\pi_0(Z(\leftexp{L}{G_{ad}}))\simeq H^1(k,G_{ad})$, this character determines a class $\omega_{\rho} \in H^{1}(k,G_{ad}).$ Using the correspondence $\rho \rightarrow \omega_{\rho}$ we distribute the representations in the $L$-packet $\Pi(\phi)$ between $GSp_4$ and $GU_2(D).$
Set
\bdm \Pi(\phi)=\coprod_{\omega\in H^{1}(k, G_{ad})}\Pi(\phi,\omega) \edm 
where \[\Pi(\phi, \omega)=\{\pi(\phi,\rho):\rho\in\mathrm{Irr}(B_{\phi}), \omega_{\rho}=\omega\}.\]
We now describe the representations $\pi(\phi,\rho).$




\subsection{DeBacker-Reeder construction}


Here we briefly review the DeBacker and Reeder construction for a quasi-split $k$-group $G$ such that $G$ is $K$-split.
Denote by $G=G(K)$ the $K$-rational points of $G$ and  $G^F=G(k)$ the $k$-rational points of $G,$ where $F$ is the Frobenius automorphism of $G$. Also, identify an $\mf{f}$-group $\ms{G}$ with its group of $\bar{\mf{f}}$-rational points. Denote $\ms{G}^F=\ms{G}(\mf{f}).$
Let $T$ be a maximal $k$-torus of $G$ such the $T$ is $K$-split. 
The dual group of $T$ is the complex torus $\hat{T}=X\otimes\mathbb{C}^{\times}$ which is a maximal torus in $\hat{G}$ satisfying
\[X^{*}(\hat{T})=X_*(T)=X^\vee,\quad  X_{*}(\hat{T})=X^{*}(T)=X.\]
Let \[\mc{A}=X^\vee\otimes\mb{R}\subset \mc{B}(G)\] be the apartment in the Bruhat-Tits building of $G$ determined by $T$. 
Denote by $W$ the affine Weyl group of $T$ in $G$, where $W\simeq N_G(T)/\leftexp{0}{T}$ for $\leftexp{0}{T}$ the maximal bounded subgroup of $T$.
For $j:G\rightarrow G_{ad}$ the adjoint quotient, let 
\[X^\vee_{ad}=X_*(j(T)),\quad \mc{A}_{ad}=X^\vee_{ad}\otimes\mb{R},\] and let $W_{ad}$ be the affine Weyl group of $j(T)=T_{ad}$ in $G_{ad}$. We also write 
\[j:X^\vee \rightarrow X^\vee_{ad},\quad j:W\rightarrow W_{ad}\] for the maps induced by $j.$
Denote by $\vartheta$ the automorphisms of $X^\vee, X_{ad}^\vee, \mc{A}, \mc{A}_{ad}, W, W_{ad}$ induced by $F.$
Choose a $F$-fixed hyperspecial vertex $o\in \mc{A}_{ad}.$

Let $\phi:W'_k\rightarrow \langle \hat{\vartheta}\rangle\ltimes\hat{G}$ be a TRD parameter for $G$.
As $\phi$ is regular, $\phi(\mr{Frob})=\hat{\vartheta}f$ where $f\in N_{\hat{G}}(\hat{T}).$ 
Denote by $\hat{w}$ the image of $f$ in  $\hat{W}_o=N_{\hat{G}}(\hat{T})/\hat{T}$.
For any $\hat{\sigma}\in\mathrm{Aut}(X)$, we define $\sigma \in \mathrm{Aut}(X^\vee)$ by $\langle\eta,\sigma\lambda\rangle = \langle\hat{\sigma}\eta,\lambda\rangle,  \eta \in X, \lambda \in X^\vee.$
Therefore, $\hat{w}$ induces a dual automorphism $w$ of $X$, where $w\in W_o=N_G(T)/T.$ 
Given $\lambda\in X^\vee$, let $t_{\lambda}\in W$ be the corresponding translation given by $x\mapsto \lambda + x.$ As an automorphism of $X^\vee,$ $w$ is a linear transformation on $\mathcal{A},$ so define 
\[\sigma_{\lambda}=t_{\lambda} w\vartheta\in  W\rtimes \langle \vartheta\rangle.\]

Since $\phi$ is a discrete series parameter we have $(X^\vee_{ad})^{w\vartheta}=\{0\}.$
The operator $I-w\vartheta$ acts invertibly on $\mc{A}_{ad},$ 
so $\sigma_{\lambda}$ has a unique fixed point there, namely 
\bdm x_{\lambda}=(I-w\vartheta)^{-1}t_{j\lambda}\cdot o.\edm 
Denote by $\tilde{x}_\lambda$ the preimage of $x_\lambda$ in $\mc{A}^{\sigma_\lambda}.$ 
By [DR, 4.4.1] we have $\mc{A}^{\sigma_\lambda}=\tilde{x}_\lambda.$
Let $J_{\lambda}$ be the unique facet of $\mathcal{A}$ containing $\tilde{x}_{\lambda},$ as in [DR, 2.7]. 
Let $K_{\lambda}$ be the parahoric subgroup of $G$ given by $J_\lambda$, and $\ms{G}_{\lambda}=K_{\lambda}/K^+_{\lambda}$ where $K^+_{\lambda}$ is the pro-unipotent radical of $K_{\lambda}.$

We have that $C_{\hat{G}}(Im(\phi))=\hat{T}^{\widehat{w\vartheta}}.$
Recall $A_\phi=\pi_0(C_{\hat{G}}(Im(\phi))).$ Any $\lambda \in X^\vee$ determines a character $\rho_{\lambda} \in \mathrm{Irr}(A_{\phi})$ by the restriction map $X^\vee \rightarrow\mr{Hom}(\hat{T}^{\widehat{w\vartheta}},\mb{C}^{\times})$ which induces an isomorphism [DR, 4.1]
\bdm [X^\vee/(1-w\vartheta)X^\vee]_{\mathrm{tor}} \tilde{\longrightarrow} \mathrm{Irr}(A_{\phi}),\quad \lambda \rightarrow \rho_{\lambda}.\edm 
Let $X_w$ be the preimage in $X^\vee$ of $[X^\vee/(1-w\vartheta)X^\vee]_{\mr{tor}}.$
For $\lambda \in X_w$,  define $[\overline{\lambda}]$ as the image of $[\lambda]$ under the map
\[ [X^\vee/(1-w\vartheta)X^\vee]_{\mr{tor}}\rightarrow [(X^\vee/\mb{Z}\Phi^\vee)/(1-\vartheta)(X^\vee/\mb{Z}\Phi^\vee)]_{\mr{tor}}=\mr{Irr}[\pi_0(\hat{Z}^{\hat{\vartheta}})]  \simeq H^1(k, G),\]
where  the first map is projection and the last map is Kottwitz' isomorphism [DR, 2.4-2.6].

Choose an alcove $C_\lambda$ in $\mc{A}$ that contains $J_\lambda$ in its closure.
There is a unique element $w_{\lambda}\in W_{\lambda}$ such that $\sigma_{\lambda}\cdot C_\lambda=w_{\lambda}\cdot C_\lambda,$ where $W_{\lambda}$ is the subgroup of $W$ generated by reflections in the hyperplanes containing $J_{\lambda}.$ If we let $y_{\lambda}=w_{\lambda}^{-1}t_{\lambda}w,$ then we have two expressions for $\sigma_{\lambda},$ namely 
\bdm t_{\lambda}w\vartheta=\sigma_{\lambda}=w_{\lambda}y_{\lambda}\vartheta.\edm 
By [DR, 2.7], there exists a lift $u_\lambda\in N_G(T)$ of $y_\lambda$ such that 
\bdm [u_\lambda]=[\overline{\lambda}]\in H^1(k, G).\edm
DeBacker and Reeder define \bdm F_\lambda = \mr{Ad}(u_\lambda)\circ F.\edm 
This determines the inner form $G^{F_\lambda}$ on which the representation they construct will live.


Choose a lift $\dot{w}$ of $w$ to an $F$-stable element of $N_G(T)\cap K_o$ where $o$ is the fixed hyperspecial vertex in $\mc{A}_{ad}.$ 
Let  $F_w=\mr{Ad}(\dot{w})\circ F.$
By [DR, 2.3.1] there exists an element $p_\lambda\in G_\lambda$ such that $p_\lambda\cdot x_\lambda=x_\lambda$ and if $T_\lambda :=\mr{Ad}(p_\lambda)T$ then
$\mr{Ad}(p_\lambda)$ intertwines $(T, F_w)$ with $(T_\lambda, F_\lambda).$
Since $p_\lambda$ fixes $x_\lambda$ we have \[\leftexp{0}{T_\lambda^{F_\lambda}}\subset K_\lambda^{F_\lambda}\subset G^{F_\lambda}.\]

We slightly modify $\phi$ to obtain a parameter $\phi'$ of $T^{F_w}$.
The $L$-group of $T^{F_w}$ is \bdm \leftexp{L}{(T^{F_w})}=\langle \hat{w}\rangle\ltimes\hat{T}.\edm 
We have that the inclusion $\hat{T}\hookrightarrow\hat{G}$ induces a bijection $\hat{T}/(1-\hat{w})\hat{T} \rightarrow (\hat{G}/\hat{G}')/(1-\hat{\vartheta})(\hat{G}/\hat{G}')$ where $\hat{G}'$ is the derived group of $\hat{G}.$
Define $\phi':W_k\longrightarrow \leftexp{L}{(T^{F_w})}$ by 
\bdm\phi'(\mathcal{I})=\phi(\mathcal{I}), \quad\phi'(\mr{Frob})=\hat{w}\ltimes u \in \langle \hat{w}\rangle\ltimes\hat{T},\edm
where $u\in\hat{T}$ is any element whose class in $\hat{T}/(1-\hat{w})\hat{T}$ corresponds to the image of $f$ (where $\phi(\mr{Frob})=\hat{\vartheta}\ltimes f$) in $(\hat{G}/\hat{G}')/(1-\hat{\vartheta})(\hat{G}/\hat{G}').$ 
As given in [DR, 4.3], by the Langlands correspondence for unramified tori, the parameter $\phi'$ determines a depth-zero character $\chi_\phi$ of $T^{F_w}.$

We have that $\leftexp{0}{T}_{\lambda}=T_{\lambda}\cap G_{\lambda}$ and that via the reduction mod $\mf{p}$ map, $\leftexp{0}{T}_\lambda$ projects onto $\ms{T}_\lambda$, an $F_\lambda$-minisotropic maximal torus in $\ms{G}_{\lambda}.$ 
Conjugate the character $\chi_{\phi}$ of $T^{F_w}$ to get a character $\chi_\lambda$ of $T_{\lambda}^{F_\lambda}$:
\bdm \chi_{\lambda}=\chi_{\phi}\circ\mathrm{Ad}(p_{\lambda})^{-1}. \edm 
The restriction of $\chi_{\lambda}$ to $\leftexp{0}{T}_{\lambda}^{F_\lambda}$ factors through a character $\chi_{\lambda}^{0}\in \mathrm{Irr}(\mathsf{T}_{\lambda}^{F_\lambda})$ which is in general position. By Deligne-Lusztig [DL] induction we have an irreducible cuspidal representation
$\epsilon_{\ms{T}}\epsilon_{\ms{G}} R_{\ms{T}_{\lambda}, \chi_{\lambda}^{0}}^{\ms{G}_{\lambda}}$ of $\ms{G}_\lambda^{F_\lambda}.$
Define 
\bdm \pi_\lambda=c-\mr{Ind}_{Z^F K_{\lambda}^{F_\lambda}}^{G^{F_\lambda}}\big(\chi_{\lambda}\otimes \epsilon_{\ms{T}}\epsilon_{\ms{G}} R_{\ms{T}_{\lambda}, \chi_{\lambda}^{0}}^{\ms{G}_{\lambda}}\big).\edm 
By [DR, 4.5.1],  $\pi_\lambda\in[\pi(\phi,\rho)]$ 
is an irreducible supercuspidal representation of $G^{F_\lambda}.$



\subsection{Extending DeBacker-Reeder}

We extend the DeBacker-Reeder construction to inner forms of an unramified group $G$ in the situation where $H^1(K, Z)=0,$ where $Z$ is the center of $G$. In this case
\[G(K)\rightarrow G_{ad}(K)\rightarrow 1\] is surjective where $j:G\rightarrow G_{ad}$ is the adjoint quotient. The notation is the same as in the previous section. 

Let $\phi$ be a TRD parameter for $G$.
We have $\widehat{T_{ad}}=\hat{T}\cap \widehat{G_{ad}}$ is a maximal torus of $\widehat{G_{ad}}.$
Since \bdm N_{\hat{G}}(\hat{T})/\hat{T}\cong N_{\widehat{G_{ad}}}(\widehat{T_{ad}})/\widehat{T_{ad}}\edm 
are canonically isomorphic, let $\hat{w}$ also denote the corresponding element in $N_{\widehat{G_{ad}}}(\widehat{T_{ad}})/\widehat{T_{ad}}.$
Denote by $w$ the corresponding automorphism of $X^\vee_{ad}.$
We apply the DeBacker-Reeder construction as given above to the group $G_{ad}.$ 
Let $\lambda\in X^\vee_{ad}.$
We have a choice of element $\dot{w}\in N_{G_{ad}(k)}(T_{ad}(k))\cap K_{o,ad}$ where $K_{o,ad}$ is the parahoric subgroup attached to the $F$-fixed hyperspecial vertex $o\in\mc{A}_{ad}$.
The DeBacker-Reeder construction also gives us an element  $p_\lambda\in G_{ad}(K)$ such that the map $T_{ad}\rightarrow \mr{Ad}(p_\lambda)T_{ad}$ satisfies 
$F_\lambda\circ\mr{Ad}(p_\lambda)=\mr{Ad}(p_\lambda)\circ F_w.$

We will use this data to construct representations of inner forms of $G$.
Let $X_w$ be the preimage in $X^\vee$ of $[X^\vee/(1-w\vartheta)X^\vee]_{\mr{tor}}$, and let $X_{w,ad}$ be the preimage in $X^\vee_{ad}$ of $[X^\vee_{ad}/(1-w\vartheta)X^\vee_{ad}]_{\mr{tor}}$.
We have $X^\vee\rightarrow X^\vee_{ad}\rightarrow 0,$ and for $\lambda\in X_{w,ad},$ there is a unique lift $\dot{\lambda}$ of $\lambda$ to $X_w$. Let \bdm \sigma_{\dot{\lambda}}=t_{\dot{\lambda}} w\vartheta.\edm
We also have $\mc{A} \longrightarrow\mc{A}_{ad}\longrightarrow 0$. Let $\tilde{x}_\lambda$ be the pre-image of $x_\lambda$ in $\mc{A}^{\sigma_{\dot{\lambda}}}$. Let $J_\lambda$ be the unique facet containing $\tilde{x}_\lambda,$ as in [DR, 2.7]. Let $K_\lambda$ be the parahoric subgroup of $G(K)$ determined by $J_\lambda,$ and let $\ms{G}_\lambda=K_\lambda/K_\lambda^+.$  
The class $[u_\lambda]\in H^1(k, G_{ad})$ determines the inner form of $G$ containing $K_\lambda.$ 

Take the element $p_\lambda\in G_{ad}(K)$ and pull it back to an element we also denote $p_\lambda\in G(K)$. We will use the notation $p_{\lambda,sc}$ if we need to distinguish the pullback from its image in $G_{ad}(K).$ Any two such choices of pullback $p_\lambda$ differ by an element of the center of $G(K)$. 
Denote also by $\dot{w}$ an element of $K^F_o$, where $K_o$ is the parahoric subgroup of $G(K)$ attached to the hyperspecial vertex $o$ in $\mc{A}_{ad}$, that projects onto $\dot{w}$ as given above. Let \bdm F_w=\mr{Ad}(\dot{w})\circ F.\edm
The parameter $\phi$ determines a character $\chi_\phi$ of $T^{F_w}.$  Define
\bdm T_\lambda:= \mr{Ad}(p_\lambda)T.\edm 
We have $G(K)$ acts on $\mc{B}(G_{ad}(K))$ via the map $j$. By [DR, 4.4.1], $p_\lambda\cdot x_\lambda=x_\lambda$. Therefore $p_{\lambda,sc}\cdot x_\lambda=x_\lambda$ so that \bdm \leftexp{0}{T}_\lambda\subset K_\lambda\edm where $\leftexp{0}{T}_\lambda$ is the maximal compact subgroup of $T_\lambda.$

Conjugate the character $\chi_{\phi}$ of $T^{F_w}$ to get a character $\chi_\lambda$ of $T_{\lambda}^{F_\lambda}:\,\chi_{\lambda}=\chi_{\phi}\circ\mathrm{Ad}(p_{\lambda})^{-1}.$
Define 
\bdm \pi_\lambda=c-\mr{Ind}_{Z^F K_{\lambda}^{F_\lambda}}^{G^{F_\lambda}}\big(\chi_{\lambda}\otimes \epsilon_{\ms{T}}\epsilon_{\ms{G}} R_{\ms{T}_{\lambda}, \chi_{\lambda}^{0}}^{\ms{G}_{\lambda}} \big)\edm where $\pi_\lambda\in[\pi(\phi,\rho_\lambda)].$

\begin{lemma} The representation $\pi_\lambda$ of $G^{F_\lambda}$ is irreducible supercuspidal. \end{lemma}

\begin{proof} The proof is the same as [DR, 4.5.1]. \end{proof}

\begin{lemma} The $G(K)$-orbit $[u_\lambda,\pi_\lambda]=\mr{Ad}(G(K))\cdot (u_\lambda,\pi_\lambda)$ depends only on the character $\rho_\lambda\in\mr{Irr}(B_\phi).$ \end{lemma}

\begin{proof}
Let $X_{w,ad}$ denote the preimage in $X^\vee_{ad}$ of $[X^\vee_{ad}/(1-w\vartheta)X^\vee_{ad}]_{\mr{tor}}\simeq \mr{Irr}(B_\phi)$. 
Given a TRD parameter $\phi$ and $\lambda\in X_{w,ad}$,
using the DeBacker Reeder construction applied to $G_{ad}$ choices of $C_\lambda, u_\lambda, p_\lambda, \dot{w}$ were made.
In defining the representation $\pi_\lambda$ we made choices of lifts $\dot{w}_{sc}$ and $p_{\lambda,sc}$ in $G(K)$.
Given a TRD parameter $\phi$, for $\lambda,\mu\in X_{w,ad},$ make choices 
\bdm (C_\lambda, u_\lambda, p_{\lambda,sc},\dot{w}_{\lambda,sc}),\quad (C_\mu, u_\mu, p_{\mu,sc}, \dot{w}_{\mu,sc})\edm 
respectively. We will show that $\rho_\lambda=\rho_\mu$ if and only if there exists a $g\in G(K)$ such that
\bdm g*u_\lambda=u_\mu, \quad g\cdot J_\lambda=J_\mu, \quad \mr{Ad}(g)_*\kappa_\lambda\simeq\kappa_\mu.\edm
This is what we mean by the statement the $G(K)$-orbit $[u_\lambda,\pi_\lambda]$ depends only on $\rho_\lambda.$
By [MP, 6.2] those three conditions are equivalent to having a $g\in G$ such that $\mr{Ad}(g)\cdot (u_\lambda, \pi_\lambda)=(u_\mu,\pi_\mu).$

By [DR, 4.5.2] there exists a $g_{ad} \in G_{ad}(K)$ such that 
\bdm g_{ad}*u_\lambda=u_\mu, \quad g_{ad}\cdot x_\lambda=x_\mu, \quad \mr{Ad}(g_{ad})(T_{ad})_\lambda =\mr{Ad}(s_{ad}) (T_{ad})_\mu, \edm
where 
$s_{ad}\in (T_{ad})_\mu.$ 
As $G(K)\rightarrow G_{ad}(K)$ is surjective, choose a lift $g$ of $g_{ad}$ to $G(K)$.
Then \bdm g*u_\lambda=g_{ad}*u_\lambda=u_\mu.\edm
An element $g\in G(K)$ acts on $\mc{B}(G_{ad}(K))$ via the adjoint quotient $j:G(K)\rightarrow G_{ad}(K),$ so 
\bdm g\cdot x_\lambda=g_{ad}\cdot x_\lambda=x_\mu.\edm 
Therefore $g\cdot J_\lambda=J_\mu.$
Also choose a lift $s$ of $s_{ad}$ to $G(K).$
Then, \bdm \mr{Ad}(g)T_\lambda=\mr{Ad}(s)T_\mu, \quad \chi_\lambda\circ\mr{Ad}(g)^{-1}=\chi_\mu\circ\mr{Ad}(s)^{-1},\edm
where $s\in T_\mu.$
This shows $\mr{Ad}(g)_*\kappa_\lambda\simeq\kappa_\mu.$ 
\end{proof}

\begin{lemma} Given a TRD parameter $\phi$ for $G$ and $\rho_\lambda\in\mr{Irr}(B_\phi),$ let $[u_\lambda,\pi_\lambda]$ be the associated $G(K)$-orbit of representations. For any $\rho_\lambda\in \mr{Irr}(A_\phi)\subset\mr{Irr}(B_\phi)$,  let $[u_\lambda,\pi_\lambda]_{DR}$ be the $G(K)$-orbit of representations associated to $\phi$ and $\rho_\lambda$ by DeBacker and Reeder. Then, \bdm [u_\lambda,\pi_\lambda]=[u_\lambda,\pi_\lambda]_{DR}.\edm \end{lemma}

\begin{proof}
Let $\phi$ be a $TRD$ parameter for $G$ and let $\rho_\lambda \in\mr{Irr}(A_\phi).$ One shows that for each step in the DeBacker and Reeder construction of a representation $\pi_\lambda$ associated to $\rho_\lambda$, one can make concurrent choices in our construction so that \bdm \pi_\lambda\in[u_\lambda,\pi_\lambda]_{DR}\,\,\Longleftrightarrow\,\, \pi_\lambda\in[u_\lambda,\pi_\lambda].\edm
\end{proof}

\subsection{ Construction of $L_\phi^{DR}$ for $GSpin_5$}

In this section let $G=GSpin_5$. Then $\hat{G}=GSp_4$, $G_{ad}=SO_5$, and $\widehat{G_{ad}}=Sp_4$.  
The root datum $(X,\Phi,X^\vee, \Phi^\vee)$ of $GSpin_5$ can be described as follows [AS, Prop 2.1]. 
We have that
\bdm X=\mb{Z}e_0\oplus\mb{Z}e_1\oplus\mb{Z}e_2, \quad
X^\vee=\mb{Z}e_0^*\oplus\mb{Z}e_1^*\oplus\mb{Z}e_2^*,\edm 
\bdm \Delta=\{a_1=e_1-e_2, a_2=e_2\},\quad
\Delta^\vee=\{a_1^\vee=e_1^*-e_2^*, a_2^\vee=2e_2^*-e_0^*\}\edm
are the character and cocharacter lattices, and simple roots and coroots, respectively.
Then \bdm T(K)=\{\prod_{j=0}^{2} e_j^*(\lambda_j):\, \lambda_j\in K^\times\}\edm is the set of $K$ points of a maximal $K$-split torus $T$ of $G$. 

Let $T_{ad}=T/Z$ where $Z$ is the center of $G$. Then $X^*(T_{ad})\hookrightarrow X^*(T)$ is the submodule $\bigoplus a_ie_i$ such that $a_0=0.$ This submodule contains the simple roots $e_1-e_2$ and $e_2.$ Also, \bdm X_*(T)\rightarrow X_*(T_{ad})=\big(\bigoplus_i \mb{Z}e_i^*\big)/ \mb{Z}e_0^*.\edm  The simple coroots for $T_{ad}$ are $e_1^*-e_2^*$ and $2e_2^*$, the image of the coroots for $T$. 
We see that \bdm X_{ad}=\mb{Z}e_1\oplus\mb{Z}e_2,\quad X_{ad}^\vee=\mb{Z}e_1^*\oplus\mb{Z}e_2^*,\quad \Delta=\{ e_1-e_2, e_2\},\quad \Delta^\vee=\{ e_1^*-e_2^*, 2e_2^*\}\edm gives a root datum for $SO_5.$ 
We have \bdm\mc{A}=\mb{R}e_0^*\oplus\mb{R}e_1^*\oplus\mb{R}e_2^*, \quad \mc{A}_{ad}=\mathbb{R}e_1^*\oplus\mathbb{R}e_2^*,\edm 
 where the projection $j:\mc{A}\rightarrow\mc{A}_{ad}$ is given by $(x_0 ,x_1 , x_2 )\mapsto (x_1, x_2 ).$
The fundamental chamber $C$ is given by the inequalities \bdm 1-e_2 >e_{1}>e_{2}>0.\edm 

We have the short exact sequence 
\[1\rightarrow GL_1 \rightarrow GSpin_5 \rightarrow SO_5 \rightarrow 1\] which gives
\bdm 1\rightarrow GL_1(K) \rightarrow GSpin_5(K) \rightarrow SO_5(K) \rightarrow \mr{H}^1(K, GL_1)=0 \edm so that \[GSpin_5(K)\rightarrow SO_5(K)\rightarrow 1\] is surjective. We apply the construction of the previous section.

Let $\phi$ be a TRD parameter for $GSpin_5.$ 
If $\phi$ is irreducible, $\hat{w}$ is a Coxeter element. 
Precisely, if $s_1$ is the simple reflection corresponding to the simple root $e_1-e_2$ for $GSp_4$ and  $s_2$ is the simple reflection corresponding  to the simple root $2e_2-e_0$ for $GSp_4,$ then 
$\hat{w}=s_1s_2 $ or $s_2s_1.$
If $\phi=\phi_1\oplus\phi_2$ then $\hat{w}=s_1s_2s_1s_2.$
If $\phi$ is irreducible, or $\phi=\phi_1\oplus\phi_2$ respectively, as automorphisms of $X$,
\begin{align*} &\hat{w}e_{0}=e_{0}+e_2,\quad \hat{w}e_{1}=-e_2,\quad \hat{w}e_{2}=e_1,\\
\hat{w}e&_{0}=e_{0}+e_{1}+e_2,\quad \hat{w}e_{1}=-e_1,\quad \hat{w}e_{2}=-e_2.\end{align*}
The automorphism $w \in\mr{Aut}(X^\vee_{ad})$ is defined by the equations $\langle e_{i}, we_{j}^*\rangle=\langle\hat{w}e_{i}, e_{j}^*\rangle$ where $i,j=1,2.$ Then, 
\begin{align*} &\,we_{1}^*=e_2^*,\quad we_{2}^*=-e_1^*,\quad \phi\,\,\,\mr{irreducible}\\
&we_{1}^*=-e_{1}^*,\quad we_{2}^*=-e_2^*,\quad\phi=\phi_1\oplus\phi_2.\end{align*}
One can compute \begin{align*} &[X_{ad}^\vee/(1-w)X_{ad}^\vee]_{\mathrm{tor}}=\{\overline{0},\overline{e_1^*}=\overline{e_2^*}\}\cong\mr{Irr}(B_\phi),\quad \phi\,\,\,\mr{irreducible}\\
[X&_{ad}^\vee/(1-w)X_{ad}^\vee]_{\mathrm{tor}}=\{\overline{0},\overline{e_{1}^*+e_{2}^*}, \overline{e_1^*}, \overline{e_2^*}\}\cong\mr{Irr}(B_\phi),\quad\phi=\phi_1\oplus\phi_2.\end{align*}

Let $\rho_\lambda\in [X_{ad}^\vee/(1-w)X_{ad}^\vee]_{\mathrm{tor}}$. We obtain an element $x_\lambda\in \mc{A}_{ad}$ and a facet $J_\lambda\in \mc{A}.$
If $K_\lambda$ is the parahoric subgroup of $GSpin_5(K)$ determined by $J_\lambda,$ we can determine the root datum of the group
$\ms{G}_\lambda:=K_\lambda/K_\lambda^+$ as follows.
Let $J$ be a facet in the fundamental chamber $C$. 
The reduction mod $\mf{p}$ of $T$, denoted $\ms{T}$, is a maximal $\bar{\mf{f}}$-split torus of the reduction mod $\mf{p}$ of $G^{J}.$ The character group of $\ms{T}$ is canonically isomorphic with the character group of $T$.
The positive roots of $\mathsf{G}_J$ are 
\bdm\Phi^{+}_{J}=\{a\in\Phi^{+}:\langle a, x \rangle\in \mb{Z}\textrm{ for all } x\in J\},\edm where $\Phi^+$ is a set of positive roots for $G$ given by a choice of simple roots $\Delta$.
The coroot associated with a root $a\in\Phi^+_J$ is the same for $\ms{G}_J$ as for $G$.
See the appendix for the determination of root datum for specific connected reductive linear algebraic groups that occur in this paper. 

The affine Weyl group $W_{ad}$ decomposes as $W_{ad}=\Omega_CW^\circ,$ where $W^\circ$ acts simply transitively on the alcoves in $\mc{A}_{ad}$ and $\Omega_C=\{\omega\in W\,:\, \omega\cdot C=C\}.$
Up to conjugation to an automorphism of the fundamental chamber $C\subset\mc{A}_{ad}$ the element $y_\lambda\in \Omega_C.$
 For $G_{ad}=SO_5,$ we have $\Omega_C=\{1,-1\}$ where $-1$ acts on $\mc{A}_{ad}$ as reflection in the first coordinate \[ -1\cdot(x_1, x_2 )=(1-x_1 ,x_2 ).\]
The class $[u_\lambda]\in H^1(k, G_{ad})$, where $u_\lambda\in N_{G_{ad}(K)}(T_{ad}(K))$ is a lift of $y_\lambda$, determines the inner form of $GSpin_5$ containing $K_\lambda.$ 
We have the following:

\renewcommand{\arraystretch}{1.6}
\begin{center} \begin{tabular}{| c | c | c | c | c |} 
\hline
$\phi$ & $\rho_\lambda$ & $x_\lambda$ & $\Phi_{x_\lambda}^+$ & $[u_\lambda]$  \\
\hline 
irred & $\overline{0}$ & 0 & $\{e_1-e_2, e_2, e_1, e_1+e_2\}$ & 1 \\
\hline
irred & $\overline{e_1^*}$ & $1/2(e_1^*+e_2^*)$ & $\{e_1-e_2, e_1+e_2\}$ & $-1$ \\
\hline
$\phi_1\oplus\phi_2$ & $\bar{0}$ & 0 & $\{e_1-e_2, e_2, e_1, e_1+e_2\}$ & 1 \\
\hline
$\phi_1\oplus\phi_2$ & $\overline{e_1^*+e_2^*}$ & $1/2(e_1^*+e_2^*)$ & $\{e_1-e_2, e_1+e_2\}$ & 1 \\
\hline
$\phi_1\oplus\phi_2$ & $\overline{e_1^*}$ & $1/2(e_1^*)$ & $\{e_2\}$ & $-1$ \\
\hline
$\phi_1\oplus\phi_2$ & $\overline{e_2^*}$ & $1/2(e_2^*)$ & $\{e_1\}$ & $-1$ \\
\hline
 \end{tabular} \end{center}
\vskip 10pt

In the following table we view $\rho_\lambda\in\mr{Irr}(B_\phi).$
Note that \bdm \leftexp{2}{GSpin}_4\cong GL_2(\mb{F}_{q^2})^\circ=\{g\in GL_2(\mb{F}_{q^2}): \,\mr{det}(g)\in \mb{F}^\times_q\}, \quad \leftexp{2}{GSpin}_2\cong \mb{F}^\times_{q^2}.\edm 
We have:

\begin{center} \begin{tabular}{| c | c | c | c | c | } 
\hline
$\phi$ & $\rho_\lambda$   & $\ms{G}_\lambda^{F_\lambda}$  & $G^{F_\lambda}$  \\
\hline 
irred & $(1)$  & $GSpin_5(\mf{f})$ & $GSpin_5(k)$ \\
\hline
irred & $(-1)$  & $\leftexp{2}{GSpin}_4(\mf{f})$ & $GSpin_{4,1}(k)$ \\
\hline
$\phi_1\oplus\phi_2$ & $(1,1)$ & $GSpin_5(\mf{f})$ & $GSpin_5(k)$ \\
\hline
$\phi_1\oplus\phi_2$ & $(-1,-1)$  & $GSpin_4(\mf{f})$ &  $GSpin_5(k)$ \\
\hline
$\phi_1\oplus\phi_2$ & $(-1,1)$  & $[(\leftexp{2}{GSpin}_2\times GSpin_3)/\Delta GL_1](\mf{f})$ & $GSpin_{4,1}(k)$  \\
\hline
$\phi_1\oplus\phi_2$ & $(1,-1)$  & $[(\leftexp{2}{GSpin}_2\times GSpin_3)/\Delta GL_1](\mf{f})$ & $GSpin_{4,1}(k)$  \\
\hline
 \end{tabular} \end{center}
\vskip 10pt

The following table lists the data for the irreducible cuspidal Deligne-Lusztig representation $\epsilon_{\ms{T}}\epsilon_{\ms{G}} R_{\ms{T}_{\lambda}, \chi_{\lambda}^{0}}^{\ms{G}_{\lambda}}$ of $\ms{G}_\lambda^{F_\lambda}.$ The notation $+,-,\dag,\ddag$ is introduced in Section 5.3 and will be used in the remainder of the paper to denote the various cases.
The character $\chi_\lambda$ of $\ms{T}_\lambda^{F_\lambda}$ is given by an element $t$ in the $F^*_\lambda$-stable dual torus $\ms{T^*_\lambda}^{F^*_\lambda}.$ 
We give the element $s$ in the $F^*_w$-stable torus $\ms{T^*}^{F^*_w}$ such that $t$ is conjugate to $s$.
We have \bdm \tau\in \mb{F}_{q^4}^\times\setminus\mb{F}_{q^2}^\times,\edm and 
\bdm \tau_1,\tau_2\in\mb{F}_{q^2}^\times\setminus\mb{F}_q^\times,\quad N_{\mb{F}_{q^2}^\times/\mb{F}_q^\times}(\tau_1)=N_{\mb{F}_{q^2}^\times/\mb{F}_q^\times}(\tau_2)\quad\mr{and}\quad \tau_1\ne\tau_2,\tau_2^q.\edm


\renewcommand{\arraystretch}{1.6}
\begin{center}\begin{tabular}{| c | c | c | c | c |}
\hline
$\phi$ & $\pm R_{\ms{T}_\lambda^*}(t)$  & $s$  & $\ms{G}_\lambda^{F_\lambda}$ \\
\hline
irred & $R_\pi^+$    & $\mr{diag}(\tau,\tau^q,\tau^{q^3},\tau^{q^2})$ & $GSpin_5(\mf{f})$ \\
\hline  
irred & $R_\pi^\dag$  &  diag$(\tau,\tau^q,\tau^{q^3},\tau^{q^2})$  & $\leftexp{2}{GSpin}_4(\mf{f})$  \\
  \hline
$\phi_1\oplus\phi_2$ & $R_\pi^+$  & diag$(\tau_1,\tau_2,\tau_2^q,\tau_1^q)$  &  $GSpin_5(\mf{f})$  \\
 \hline
$\phi_1\oplus\phi_2$ & $R_\pi^-$  & diag$(\tau_1,\tau_2,\tau_2^q,\tau_1^q)$ & $GSpin_4(\mf{f})$  \\
 \hline
$\phi_1\oplus\phi_2$ & $R_\pi^\ddag$  & diag$(\tau_1,\tau_2,\tau_2^q,\tau_1^q)$  & $[(\leftexp{2}{GSpin}_2\times GSpin_3)/\Delta GL_1](\mf{f})$  \\
 \hline
$\phi_1\oplus\phi_2$ & $R_\pi^\ddag$  & diag$(\tau_2,\tau_1,\tau_1^q,\tau_2^q)$  & $[(\leftexp{2}{GSpin}_2\times GSpin_3)/\Delta GL_1](\mf{f})$  \\
  \hline
\end{tabular}\end{center}

For each $\rho_\lambda\in \mr{Irr}(B_{\phi})$ we have
\bdm \pi_\lambda=c-\mr{Ind}_{Z^F K_{\lambda}^{F_\lambda}}^{G^{F_\lambda}}\big(\chi_{\lambda}\otimes \epsilon_{\ms{T}}\epsilon_{\ms{G}} R_{\ms{T}_\lambda^*}^{\ms{G}_\lambda}(t) \big)\edm where $\pi_\lambda\in[\pi(\phi,\rho_\lambda)].$ Recall
$\Pi(\phi, \omega)=\{\pi(\phi,\rho):\rho\in\mathrm{Irr}(B_{\phi}), \omega_{\rho}=\omega\}.$
Let \[L_\phi^{DR}=\Pi(\phi, 1)\quad \mr{or}\quad \Pi(\phi, -1)\] be the $L$-packet of depth zero supercuspidal representations of $GSpin_5$ or $GSpin_{4,1}$, respectively, attached to the TRD parameter $\phi$ by the DeBacker-Reeder construction.

\section{The generalized principal series $I(s,\pi\boxtimes\sigma)$}

\subsection{Tame regular discrete series local Langlands for $GL_{2m}$}

Let $\phi=\phi_1\oplus\dots\oplus\phi_r$, where $r=1,2$, be a tame regular discrete series Langlands parameter for $GSpin_5$. Then, for $1\le i\le r$,
\bdm \phi_i: W_k\longrightarrow GSp_{2m}(\mb{C})\hookrightarrow GL_{2m}(\mb{C})=GL_{2m}(k)^\vee\edm is also Langlands parameter for $GL_{2m}(k)$.
By the local Langlands correspondence for $GL_{2m}$, let \bdm \phi_i \longleftrightarrow L_{\phi_i}=\{\sigma_{\phi_i}\},\edm where
the $L$-packets for $GL_{2m}$ are always of size one.

By Section 3, $\phi_i=\mr{Ind}_{W_{k_{2m}}}^{W_k} \eta,$ where $\eta$ is a regular character of $W_{k_{2m}}.$
By [He1], the representation $\sigma_{\phi_i}$ is the irreducible depth zero supercuspidal representation $\pi(\eta)\otimes \omega,$ where $\pi(\eta)$ is the representation attached to $\phi_i$ by Gerardin [Ge], and $\omega$ is the unramified character of order two of the extension $k_{2m}$ over $k$.

Under the Artin map, the character $\eta$ of $W_{k_{2m}}$ corresponds to a character, which we will also denote $\eta$, of $k_{2m}^\times.$ As $\eta$ is tame, $\eta\vert_{\mf{o}_{k_{2m}}^\times}$ factors through a character of $\mf{f}_{2m}^\times.$ We can view $\mf{f}_{2m}^\times\subset GL_{2m}(\mf{f})$ as the group of $\mf{f}$-points of a minisotropic torus $\ms{S}$ of $GL_{2m}(\bar{\mf{f}})$ which is defined over $\mf{f}.$ As $\eta$ is regular, as a character of $\mf{f}_{2m}^\times$ it is in general position, and the Deligne-Lusztig character $\epsilon_{\ms{G}}\epsilon_{\ms{S}}R_{\mf{f}_{2m}^\times,\eta}$ gives an irreducible cuspidal representation of $GL_{2m}(\mf{f}).$
Then \[ \pi(\eta)=c-\mr{Ind}_{k^\times GL_{2m}(\mf{o})}^{GL_{2m}(k)} \eta\otimes \epsilon_{\ms{G}}\epsilon_{\ms{S}}R_{\mf{f}_{2m}^\times,\eta}.\]

Let $\ms{T}\subset GL_{2m}(\bar{\mf{f}})$ be the split maximal torus. We can view the character $\eta$ as a regular element $t$ in the dual $F^*$ stable torus  $\ms{S^*}^{F^*},$ where $t$ is conjugate over $GL_{2m}(\bar{\mf{f}})$ to an element $s$ in the Coxeter torus $\ms{T^*}^{F_w^*}.$ Here the regular Frobenius action $F$ is twisted by the Coxeter element of the Weyl group, and we have the dual action
\[F_w^*(x_1,x_2,\dots,x_{2m})=(x_{2m}^q,x_1^q,\dots,x_{2m-1}^q).\]
Given a tame regular discrete series parameter $\phi$, we list the corresponding element $s$, where $\tau,\tau_1,$ and $\tau_2$ are as in the previous section.

\renewcommand{\arraystretch}{1.6}
\begin{center}\begin{tabular}{| c | c | c | c | c |}
\hline
$\phi_i$ & $\pm R_{\ms{S}^*}(t)$  & $s$  & $GL_{2m}(\mf{f})$ \\
\hline
4 dim & $R_\sigma^{+ \mr{or} \dag}$  & diag$(\tau, \tau^q, \tau^{q^2}, \tau^{q^3})$  & $GL_4(\mf{f})$ \\
 \hline
2 dim & $R_\sigma^{+,-,\mr{or}\ddag}$  & diag$(\tau_1, \tau_1^q)$  & $GL_2(\mf{f})$ \\
\hline
2 dim & $R_\sigma^{+,-,\mr{or}\ddag}$ & diag$(\tau_2, \tau_2^q)$ & $GL_2(\mf{f})$ \\

  \hline
\end{tabular}\end{center}

\subsection{The Bernstein component of $I(s,\pi\boxtimes\sigma)$}
 
From now on, let $G=GSpin_{4m+5}(k)$ or $GSpin_{2m+4,2m+1}(k)$.
In the following, we will always assume $m=1$ or $2.$
For simplicity of notation, let $n=2m+2.$
Let $P=M\cdot N$ be the maximal parabolic subgroup of $G$ with Levi factor 
\[M=GSpin_{5}(k)\times GL_{2m}(k) \quad\mr{if}\quad G=GSpin_{4m+5}(k);\]
\[M=GSpin_{4,1}(k)\times GL_{2m}(k) \quad\mr{if} \quad G=GSpin_{2m+4,2m+1}(k).\]
Note that in each case $P$ corresponds to the simple root $a_{2m}=e_{2m}-e_{2m+1}.$
If $\pi$ is an irreducible representation of $GSpin_5(k)$ or $GSpin_{4,1}(k),$ and $\sigma$ a representation of $GL_{2m}(k)$ we can form the generalized principal series representation
\[I(s,\pi\boxtimes\sigma)=\mr{Ind}_{P}^{G} \delta_{P}^{1/2} \pi\boxtimes\sigma\vert \mr{det} \vert^{s},\] where $s\in\mb{C}.$

Let $\mathscr{B}(G)$ be the set of classes of irreducible supercuspidal representations of rational Levi components of rational parabolic subgroups of $G$ under the equivalence given from $G$-conjugation and twisting by unramified quasicharacters of the Levi components. The inertial support  of an irreducible representation of $G$  is the inertial equivalence class $\mf{s}\in\mathscr{B}(G)$ of the support of the representation.
The theory of the Bernstein centre [Be] decomposes $\mf{R}(G)$, the category of smooth complex representations of $G$, into subcategories $\mf{R}^{\mf{s}}(G)$, where the objects of $\mf{R}^{\mf{s}}(G)$ are the smooth representations of $G$ all of whose irreducible subquotients have inertial support $\mf{s}.$
Any unramified quasicharacter $\chi$ of $GSpin_5 \times GL_{2m}$ or $GSpin_{4,1}\times GL_{2m}$ is of the form
\bdm\chi=\vert\mr{sim}\vert^t\boxtimes\vert\mr{det}\vert^s,\quad s,t\in\mb{C},\edm where $\mr{sim}$ and $\mr{det}$ are $k$-rational characters of $GSpin_5$ or $GSpin_{4,1}$ and $GL_{2m}$, respectively.
Therefore, the Bernstein component of $I(s,\pi\boxtimes\sigma)$ is
the set of all smooth representationss $\kappa$ of $G$ such that all the irreducible subquotients of $\kappa$ are a composition factor of a representation equivalent to
\bdm \mr{Ind}_{P}^{G}\delta_{P}^{1/2}\pi\vert \mr{sim}\vert^t \boxtimes\sigma\vert\mr{det}\vert^s,\edm for some $s,t\in\mb{C}.$

From now on let $I(s,\pi\boxtimes\sigma)$ be the generalized principal series where given a tame regular discrete series parameter $\phi$, $\pi\in L_\phi^{DR}$ where $L_\phi^{DR}$ is the $L$-packet of representations of $GSpin_5(k)$ or $GSpin_{4,1}(k)$ as given in Section 4.5, and  $\sigma=\sigma_{\phi_i}$ as given in Section 5.1. 

In [Mo1], Morris shows that if $\mc{P}$ is a parahoric subgroup of a connected reductive $k$-group $G$, where the reductive quotient of $\mc{P}$ is denoted $\ms{M},$ and $\rho$ is an irreducible cuspidal reprentation of $\ms{M}$, then $(\mc{P},\rho)$ is a $\mf{S}$-type for a finite set of inertial equivalence classes $\mf{S}\subset\mathscr{B}(G).$  Using this work of Morris, we give a parahoric subgroup $\mc{P}$ of $G$ and an irreducible cuspidal representation $\rho$ of the reductive quotient $\ms{M},$ such that $(\mc{P},\rho)$ is a $[M,\pi\boxtimes\sigma]_G$-type in $G.$ 
Then, using the theory of types and covers developed by Bushnell and Kutzko [BK2], we have that representations in the Bernstein component of $I(s,\pi\boxtimes\sigma)$ are parametrized by unital left modules of 
$\mc{H}(G, \rho)$ the Hecke algebra of compactly supported $\rho$-spherical functions on $G.$ 
Namely,
\bdm \mc{H}(G,\rho)=\{f\in\mc{C}_c^\infty(G,\mathrm{End}_{\mathbb{C}}(W^{\vee}))\vert\,
 f(k_{1}gk_{2})=\check{\rho}(k_{1})f(g)\check{\rho}(k_{2}), \quad k_{i}\in \mc{P}, \, g\in G\}, \edm where
$\mc{C}_c^\infty(G,\mathrm{End}_{\mathbb{C}}(W^{\vee}))$ is the set of functions $f:G\rightarrow\mathrm{End}_{\mathbb{C}}(W^{\vee})$ such that $f$ is locally constant with compact support and $(\check{\rho},W^{\vee})$ denotes the contragredient representation.
As algebras, \bdm\mc{H}(G,\rho)\cong\mathrm{End}_{G}(c-\mathrm{Ind}_{\mc{P}}^{G}(\rho)).\edm

\subsection{Definition of $\mc{P}$}

The vertices of the local Dynkin diagram for $G=GSpin_{4m+5}$ are in correspondence with the set
\bdm \Pi^+=\{\alpha_{0^+}=-e_{1}-e_{2} + 1,\alpha_{1^+}=e_1-e_2,\alpha_{2^+}=e_2-e_3,\dots,\alpha_{n-1^+}=e_{n-1}-e_n,\alpha_{n^+}=e_n\}.\edm
By Bruhat-Tits theory, let $\mc{K}_{i^+}$ be the parahoric subgroup of $GSpin_{4m+5}(k)$ corresponding to the root $\alpha_{i^+}$ in the local Dynkin diagram.
Let \bdm\mc{P}^+=\mc{K}_{0^+}\cap\mc{K}_{2m^+}.\edm
Another choice of simple roots for $GSpin_9$ is
\[\Delta^-=\{e_3-e_4, e_4-e_1, e_1-e_2, e_2\}.\]
Let $\alpha_{0^-}=-e_3-e_4+1$ where $-e_3-e_4$ is the lowest root. Let $\mc{K}_{i^-}$ be the parahoric subgroup of $GSpin_9(k)$ corresponding to the root $\alpha_{i^-}$ in the local Dynkin diagram corresponding to the set \bdm\Pi^-=\{\alpha_{0^-}=-e_3-e_4+1, \alpha_{1^-}=e_3-e_4, \alpha_{2^-}=e_4-e_1, \alpha_{3^-}=e_1-e_2, \alpha_{4^-}=e_2\}.\edm 
Let \bdm\mc{P}^-=\mc{K}_{2^-}\cap\mc{K}_{4^-}.\edm


The vertices of the relative local Dynkin diagram for $G=GSpin_{2m+4,2m+1}$ are in correspondence with the set
\bdm\{\alpha_0=1-e_1, \alpha_1=e_1-e_2, \alpha_2=e_2-e_3, \dots, \alpha_{2m}=e_{2m}-e_{2m+1}, \alpha_{2m+1}=e_{2m+1}\}.\edm
Note that this is a set of simple affine roots for an affine Weyl group of type $C_{2m+1}.$ 
Another choice of simple roots for $S$, where $S$ is the maximal $k$-split torus of $GSpin_{8, 5}$ gives us another set of simple affine roots 
\bdm \Pi^\dag=\{\alpha_{0^\dag}=1-e_5, \alpha_{1^\dag}=e_5-e_1, \alpha_{2^\dag}=e_1-e_2, \alpha_{3^\dag}=e_2-e_3, \alpha_{4^\dag}=e_3-e_4, \alpha_{5^\dag}=e_4\}.\edm
Let \bdm\mc{P}^\dag=\mc{K}_{1^\dag}\cap\mc{K}_{5^\dag}.\edm
The vertices of the local Dynkin diagram for $GSpin_{6,3}$ are in correspondence with the set
\bdm \Pi^\ddag=\{\alpha_{0^\ddag}=1-e_1, \alpha_{1^\ddag}=e_1-e_2, \alpha_{2^\ddag}=e_2-e_3, \alpha_{3^\ddag}=e_3\}.\edm
Let \bdm\mc{P}^\ddag=\mc{K}_{0^\ddag}\cap\mc{K}_{2^\ddag}.\edm

Denote $\mathcal{P}=\mathcal{P}^{a},a=+,-,\dag, \ddag,$ depending on the context, so
\bdm\mathcal{Q}=\mathcal{P}\cap M=K_\lambda^{F_\lambda}\times GL_{2m}(\mf{o}).\edm
Let $\mathsf{M}=\mathsf{G}_{\lambda}^{F_\lambda}\times GL_{2m}(\mathfrak{f})$ be the reductive component of the reduction mod $\mf{p}$ of $\mathcal{P}$ which is equal to the reductive component of the reduction mod $\mf{p}$ of $\mathcal{Q}.$
If $\mc{P}=\mc{P}^a,$ denote by
\bdm \rho=R_\pi\boxtimes R_\sigma=R_\pi^a\boxtimes R_\sigma^a,\quad a=+,-,\dag,\ddag, \edm
a representation of $\ms{M}$, where the $R^a$ are given in the tables in Section 4.5 and Section 5.1.
We can view $\rho$ as a representation of $\mathcal{P}$ via inflation. We can also view $\rho$ as a representation of $\mathcal{Q}$ via inflation. We denote this representation of $\mc{Q}$ by $\rho_{M}.$

\begin{lemma} The pair $(\mathcal{Q},\rho_{M})$ is a type for the inertial class $[M,\pi\boxtimes\sigma]_{M}$ in $M$ and $(\mathcal{P},\rho)$ is a $G$-cover for it. Therefore, the pair $(\mathcal{P},\rho)$ is a type for the inertial class $[M,\pi\boxtimes\sigma]_{G}$ in $G.$ \end{lemma}

\begin{proof}
Since $((\chi_\lambda\otimes R_\pi)\boxtimes(\eta\otimes R_\sigma))\vert_{\mc{Q}}=\rho_M$ is irreducible, 
by [BK2, Prop 5.4] $(\mathcal{Q},\rho_{M})$ is a type for the inertial class $[M,\pi\boxtimes\sigma]_{M}$ in $M$.
It is shown in [Mo1, pg. 149] that $(\mathcal{P},\rho)$ is a $G$-cover for $(\mathcal{Q},\rho_{M})$. Then, by [BK2, Thm 8.3], $(\mathcal{P},\rho)$ is a type for the inertial class $[M,\pi\boxtimes\sigma]_{G}$ in $G.$ 
\end{proof}

\begin{corollary} The Hecke algebra of the Bernstein component of $I(s,\pi\boxtimes\sigma)$ is $\mc{H}(G,\rho).$
\end{corollary}

\begin{proof} We have
\bdm\mf{R}^{[M,\pi\boxtimes\sigma]_{G}}(G)=\mf{R}_{\rho}(G)\cong \mc{H}(G,\rho)\mathrm{-Mod},\edm where $[M,\pi\boxtimes\sigma]_{G}$ is the inertial support of $I(s,\pi\boxtimes\sigma).$ The last equivalence is given by the functor $M_\rho$ restricted to $\mf{R}_{\rho}(G)$.
\end{proof}

\section{The Hecke algebra $\mc{H}(G,\rho)$}

\subsection{A presentation of $\mc{H}(G,\rho)$}

In [Mo2], Morris describes explicit generators and relations for $\mc{H}(G,\rho)$ when $\rho$ is an irreducible cuspidal representation of the reductive quotient of a parahoric subgroup $\mc{P}$ of $G$. Let $\Pi$ be a set of simple affine roots in correspondence with the vertices of the relative local Dynkin diagram for $G$. 
For $\Theta\subset\Pi$ define $W_\Theta=\langle s_\alpha\,\vert\, \alpha\in\Theta\rangle.$
Define \bdm S_\Theta=\{w\in N_W(W_\Theta)\,\vert\, w\Theta=\Theta\},\edm where $N_W(W_\Theta)$ is the normalizer of $W_\Theta$ in the affine Weyl group $W$. Let $\mc{P}$ corresponds to $\Theta\subset\Pi$, and let $\ms{M}$ be the reductive quotient of $\mc{P}.$
By [Mo2, \S 4] $\mc{H}(G,\rho)$ is supported on double cosets $\mc{P}\dot{w}\mc{P}$ where $\dot{w}\in N_G(T)$
 such that under the induced action on $\ms{M}$, $\leftexp{\dot{w}}{\ms{M}}=\ms{M}$ as representations of $\ms{M}$,\,$\dot{w}\rho\cong\rho$, and $\dot{w}$ projects to an element $w\in S_\Theta.$

Suppose that $J\subset\Pi$ such that $wJ=\Theta$ for some $w\in W$, and $\alpha\in \Pi$.
Let $t$ be the longest element in the Weyl group $W_J$ corresponding to the spherical root system obtained from $J$ such that $t^2=1$ and $t(J)=-J$. Let $u$ be the longest element in the Weyl group $W_{J\cup\{\alpha\}}$ corresponding to the spherical root system obtained from $J\cup\{\alpha\}$ such that $u^2=1$ and $u(J\cup\{\alpha\})=-(J\cup\{\alpha\})$. Set \bdm v[\alpha,J]=u\cdot t\in W_{J\cup\alpha}\subset W.\edm

We now specialize to the case of interest.
Denote by $\Theta=\Theta^+\subset\Pi^+,\Theta^- \subset\Pi^-,\Theta^\dag\subset\Pi^\dag,\Theta^\ddag\subset\Pi^\ddag:$ 
\bdm \Theta^+=\{\alpha_{0^+},\dots,\alpha_{2m+2^+}\}\setminus \{\alpha_{0^+},\alpha_{2m^+}\},\quad \Theta^-=\{\alpha_{0^-},\dots,\alpha_{4^-}\}\setminus\{\alpha_{2^-},\alpha_{4^-}\},\edm
\bdm \Theta^\dag=\{\alpha_{0^\dag},\dots,\alpha_{5^\dag}\}\setminus\{\alpha_{1^\dag},\alpha_{5^\dag}\},\quad \Theta^\ddag=\{\alpha_{0^\ddag},\dots\alpha_{3^\ddag}\}\setminus\{\alpha_{0^\ddag},\alpha_{2^\ddag}\}.\edm


\begin{lemma} 
\begin{enumerate}
\item[\textup{(}i\textup{)}] $S_{\Theta^+}=\langle v[\alpha_{0^+},\Theta^+],v[\alpha_{2m^+},\Theta^+], T(e_0^*) \rangle;$ 
\item[\textup{(}ii\textup{)}] $S_{\Theta^-}=\langle v[\alpha_{2^-},\Theta^-],v[\alpha_{4^-},\Theta^-],\nu \rangle$ 
where
$\nu=T(e_3^*)s_{\alpha_{1^-}}s_{\alpha_{2^-}}s_{\alpha_{3^-}}s_{\alpha_{4^-}}s_{\alpha_{3^-}}s_{\alpha_{2^-}}s_{\alpha_{1^-}},$ 
is a diagram automorphism preserving the fundamental chamber corresponding to $\Pi^-$ such that $\nu^2=T(e_0^*)$ is translation in the central direction;
\item[\textup{(}iii\textup{)}] $S_{\Theta^\dag}=\langle v[\alpha_{1^\dag},\Theta^\dag],v[\alpha_{5^\dag},\Theta^\dag], T(e_0^*)\rangle;$
\item[\textup{(}iv\textup{)}] $S_{\Theta^\ddag}=\langle v[\alpha_{0^\ddag},\Theta^\ddag],v[\alpha_{2^\ddag},\Theta^\ddag], T(e_0^*)\rangle.$
\end{enumerate} 
\noindent In all cases the elements $v[\alpha_i,\Theta]$ are involutions, and $T(e_0^*)$ is translation by $e_0^*.$
\end{lemma}

\begin{proof}  
For $J\subset \Pi$, corresponding to a connected piece of the extended Dynkin diagram of type $B_n, C_n, D_n(n$ even), $u(\alpha_j)=-\alpha_j$ for $\alpha_j\in J$. Here, $u$ is the longest element in the Weyl group $W_J$ defined above. As the piece of the extended Dynkin diagram corresponding to $\Theta^+\cup\{\alpha_{2m^+}\}$ is of type $B_{2m+2}$, 
\bdm v[\alpha_{2m^+},\Theta^+]\Theta^+=\Theta^+.\edm  
By [Ho, Lem 10], $v[\alpha_{0^+},\Theta^+]=v[\alpha_{0^+},\{\alpha_{1^+},\dots,\alpha_{2m-1^+}\}]$, therefore for $b\in \{\alpha_{2m+1^+},\alpha_{2m+2^+}\},$ $v[\alpha_{0^+},\Theta^+](b)=b$, and for $b\in\{\alpha_{1^+},\dots,\alpha_{2m-1^+}\}, v[\alpha_{0^+},\Theta^+](b)=b$, so 
\bdm v[\alpha_{0^+},\Theta^+]\Theta^+=\Theta^+\edm $(\{\alpha_{1^+},\dots,\alpha_{2m-1^+}\}\cup\{\alpha_{0^+}\}$ corresponds to a diagram of type $D_{2m}).$ Since $v[\alpha_i,\Theta^+]\Theta^+=\Theta^+$ for $i=0^+,2m^+$, By [Mo2, Lem 2.4(c)] they are involutions.

Now, by [Mo2, Lem 2.5], if $w\in W$ such that $w\Theta=\Theta,$ we can find
$J_1,\dots,J_{r+2},$ where $J_i\subset \Pi$ and $\Theta=J_1=J_{r+2},$ and $\alpha_1,\dots\alpha_r\in\Pi$ such that  $v[\alpha_i,J_i]J_i=J_{i+1}, 1\leq i\leq r,$ and $w=\nu v[\alpha_r,J_r]\dots v[\alpha_1,J_1]$ where $\nu\in\Omega, \nu J_{r+1}=J_{r+2}=\Theta.$ Here, $\nu\in\Omega=\{w\in W\,\vert\, wC=C\}$ the set of diagram automorphisms that fix the fundamental chamber $C$.
For $GSpin_{4m+5}$, for the simple roots given by $\Delta^+$, the set $\Omega=\langle \nu\rangle$, where \bdm \nu=T(e_1^*)s_{\alpha_{1^+}}\dots
s_{\alpha_{n-1^+}}s_{\alpha_{n^+}}s_{\alpha_{n-1^+}}\dots s_{\alpha_{1^+}}\edm is a diagram automorphism preserving the fundamental chamber given by $\Delta^+$. We have $\nu^2=T(e_0^*)$ is translation in the central direction. The action of $\nu$ on the simple affine roots is as follows, \bdm \nu\cdot \alpha_{0^+}=\alpha_{1^+}, \nu\cdot\alpha_{1^+}=\alpha_{0^+}, \nu\cdot\alpha_{i^+}=\alpha_{i^+},\, i>1.\edm Therefore, $\nu$ does not preserve $\Theta^+.$ However, $\nu^2=T(e_0^*)$ does preserve $\Theta^+.$ Since $v[\alpha_i,\Theta^+]\Theta^+=\Theta^+$ for all $\alpha_i\in \Pi^+\setminus \Theta^+$, $J_1=J_2=\dots=J_{r+1}=\Theta^+$.  
 Therefore, if $w\Theta^+=\Theta^+$, $w$ is a word in $v[\alpha_{0^+},\Theta^+], v[\alpha_{2m^+},\Theta^+],$ and $T(e_0^*).$ This shows (i).

For cases (ii), (iii), (iv), the proof is as in (i).
As the piece of the extended Dynkin diagram corresponding to $\Theta^-\cup\{\alpha_{2^-}\}$ is of type $D_4$, 
$v[\alpha_{2^-},\Theta^-]\Theta^-=\Theta^-.$
By [Ho, Lem 10], $v[\alpha_{4^-},\Theta^-]=v[\alpha_{4^-},\{\alpha_{3^-}\}]$ and $v[\alpha_{4^-},\Theta^-]\Theta^-=\Theta^-$ ($\{\alpha_{3^-}\}\cup\{\alpha_{4^-}\}$ corresponds to a diagram of type $B_2$). 
As the piece of the extended Dynkin diagram corresponding to $\Theta^\dag\cup\{\alpha_{1^\dag}\}$ is of type $B_5$, 
$v[\alpha_{1^\dag},\Theta^\dag]\Theta^\dag=\Theta^\dag.$
By [Ho, Lem 10], $v[\alpha_{5^\dag},\Theta^\dag]=v[\alpha_{5^\dag},\{\alpha_{2^\dag},\alpha_{3^\dag},\alpha_{4^\dag}\}]$ and $v[\alpha_{5^\dag},\Theta^\dag]\Theta^\dag=\Theta^\dag$ ($\{\alpha_{2^\dag},\alpha_{3^\dag},\alpha_{4^\dag}\}\cup\{\alpha_{5^\dag}\}$ corresponds to a diagram of type $B_4$).
As the piece of the extended Dynkin diagram corresponding to $\Theta^\ddag\cup\{\alpha_{2^\ddag}\}$ is of type $B_3$, 
$v[\alpha_{2^\ddag},\Theta^\ddag]\Theta^\ddag=\Theta^\ddag.$
By [Ho, Lem 10], $v[\alpha_{0^\ddag},\Theta^\ddag]=v[\alpha_{0^\ddag},\{\alpha_{1^\ddag}\}]$ and $v[\alpha_{0^\ddag},\Theta^\ddag]\Theta^\ddag=\Theta^\ddag$ ($\{\alpha_{1^\ddag}\}\cup\{\alpha_{0^\ddag}\}$ corresponds to a diagram of type $B_2$).

When the simple roots are given by $\Delta^-$, the set of diagram automorphisms $\Omega=\langle \nu\rangle$, where \bdm\nu=T(e_3^*)s_{\alpha_{1^-}}s_{\alpha_{2^-}}s_{\alpha_{3^-}}s_{\alpha_{4^-}}s_{\alpha_{3^-}}s_{\alpha_{2^-}}s_{\alpha_{1^-}} \edm is a diagram automorphism preserving the fundamental chamber given by $\Delta^-$ such that $\nu^2=T(e_0^*)$ is translation in the central direction. The action of $\nu$ on the simple roots is as follows, \bdm \nu\cdot \alpha_{0^-}=\alpha_{1^-}, \nu\cdot\alpha_{1^-}=\alpha_{0^-}, \nu\cdot\alpha_{i^-}=\alpha_{i^-}, i>1.\edm Therefore in case (ii), $\nu$ does preserve $\Theta^-.$
In cases (iii) and (iv), $\Omega=\langle T(e_0^*)\rangle$ and $T(e_0^*)$ preserves both $\Theta^\dag$ and $\Theta^\ddag.$ 
\end{proof}

\begin{lemma} 
\begin{itemize}
\item[\textup{(}i\textup{)}] For $i=0^+,2m^+$ if $\Theta=\Theta^+$, $i=2^-,4^-$ if $\Theta=\Theta^-,$ $i=1^\dag,5^\dag$ if $\Theta=\Theta^\dag$, and $i=0^\ddag,2^\ddag$ if $\Theta=\Theta^\ddag,$ we have 
\bdm v[\alpha_i,\Theta]\cdot \rho\cong\rho,\edm for $\rho=R_\pi\boxtimes R_\sigma$ a representation of  $\ms{M}$ as given in Section 5.3.
\item[\textup{(}ii\textup{)}] In addition, for $i=2^-,4^-$, \bdm \nu\cdot\rho\ncong\rho,\edm for $\rho=R_\pi^-\boxtimes R_\sigma^-$ a representation of $\ms{M}=GSpin_4(\mf{f})\times GL_2(\mf{f})$. 
\end{itemize} \end{lemma}

\begin{proof} (i) For $i=0^+,2m^+$ if $\Theta=\Theta^+$, $i=2^-,4^-$ if $\Theta=\Theta^-,$ $i=1^\dag,5^\dag$ if $\Theta=\Theta^\dag$, and $i=0^\ddag,2^\ddag$ if $\Theta=\Theta^\ddag,$ $v[\alpha_i,\Theta]$ acts on the root datum for $\ms{M}$, $\Psi=(X,\Phi,X^\vee,\Phi^\vee),$ preserving the set of positive roots $\Phi^+.$ So $v[\alpha_i,\Theta]$ gives an automorphism of the based root datum for $\ms{M}$ \bdm \Psi_0=(X,\Phi^+,X^\vee,(\Phi^+)^\vee).\edm 
Let $\ms{B}$ be the Borel subgroup of $\ms{M}$ given by the positive roots $\Phi^+$, and for each $a\in\Delta$ let $u_a\ne e$ be a fixed element in the root subgroup $U_a$.
By [Sp, Prop 2.13] Aut $\Psi_0(\ms{M})$ is isomorphic to the group \bdm \mr{Aut}(\ms{M},\ms{B},\ms{T},\{u_a\}_{a\in\Delta})\edm of automorphisms of $\ms{M}$ which stabilize $\ms{B}$, $\ms{T}$ and the set of $u_a.$ Therefore to determine the automorphism of $\ms{M}$ given by $v[\alpha_i,\Theta]$, we need only find an automorphism of $\ms{M}$ that stabilizes the set of $u_a$ and gives the same action on $\ms{T}$ as $v[\alpha_i,\Theta].$ 

The action of $v[\alpha_i,\Theta^+],\,i=0^+,2m^+,$ on the maximal torus $T$ of $G=GSpin_{4m+5}(k)$ is given by
\bdm v[\alpha_i,\Theta^+]\cdot (\prod_{j=0}^{2m+2} e_j^*(\lambda_j))\edm \bdm=e_0^*(\lambda_0\lambda_1\dots\lambda_{2m})e_1^*(\lambda_{2m}^{-1})e_2^*(\lambda_{2m-1}^{-1})\dots e_{2m}^*(\lambda_1^{-1})e_{2m+1}^*(\lambda_{2m+1})\dots e_{2m+2}^*(\lambda_{2m+2}),\edm
for $\lambda_j\in GL_1(k).$ This factors through an action on $\ms{T}\subset\ms{M}.$
If $A\in GL_{2m}(\mf{f}), \,B\in GSpin_5(\mf{f})$ the action of $v[\alpha_i,\Theta^+]$ on $\ms{M}$ is given by \bdm (A,B)\mapsto (w_0(\leftexp{t}A^{-1})w_0^{-1},(\mr{det}A)B),\edm
where $w_0\in GL_{2m}$ is the matrix with 1 on the antidiagonal and 0 elsewhere.


For $i=2^-,4^-,$ $v[\alpha_i,\Theta^-]$ acts on the maximal torus $T$ of $G=GSpin_9(k)$ by
\bdm v[\alpha_i,\Theta^-]\cdot (\prod_{j=0}^4 e_j^*(\lambda_j))=e_0^*(\lambda_0\lambda_1\lambda_2)e_1^*(\lambda_2^{-1})e_2^*(\lambda_1^{-1})e_3^*(\lambda_3) e_4^*(\lambda_4).\edm
Then, if $A\in GL_{2}(\mf{f}), \,B\in GSpin_4(\mf{f})$ the action of $v[\alpha_i,\Theta^-]$ on $\ms{M}$ is given by
\bdm (A,B)\mapsto (w_0(\leftexp{t}A^{-1})w_0^{-1},(\mr{det}A)B).\edm

For $i=1^\dag,5^\dag,$ $v[\alpha_i,\Theta^\dag]$ acts on the maximal $k$-split torus $S$ of $G=GSpin_{8,5}(k)$ by
\bdm v[\alpha_i,\Theta^\dag]\cdot (\prod_{j=0}^5 e_j^*(\lambda_j))=e_0^*(\lambda_0\lambda_1\lambda_2\lambda_3\lambda_4)e_1^*(\lambda_4^{-1})e_2^*(\lambda_3^{-1})e_3^*(\lambda_2^{-1}) e_4^*(\lambda_1^{-1})e_5^*(\lambda_5).\edm 
This factors through an action on $\ms{T}\subset GL_4(\mf{f}),$ and
by [Sp, Prop 2.13], $v[\alpha_i,\Theta]$ acts on $A\in GL_4(\mf{f})$ as 
\bdm A\mapsto w_0(\leftexp{t}A^{-1})w_0^{-1}.\edm 
If $A\in GL_4(\mf{f}), \,B\in \leftexp{2}{GSpin}_4(\mf{f})$ the action of $v[\alpha_i,\Theta^\dag]$ on $\ms{M}$ is given by \bdm (A,B)\mapsto (w_0(\leftexp{t}A^{-1})w_0^{-1},(\mr{det}A)B).\edm

For $i=0^\ddag,2^\ddag,$ $v[\alpha_i,\Theta^\ddag]$ acts on the maximal $k$-split torus $S$ of $G=GSpin_{6,3}(k)$ by
\bdm v[\alpha_i,\Theta^\ddag]\cdot (\prod_{j=0}^3 e_j^*(\lambda_j))=e_0^*(\lambda_0\lambda_1\lambda_2)e_1^*(\lambda_2^{-1})e_2^*(\lambda_1^{-1})e_3^*(\lambda_3).\edm
As in the previous case we find, 
if $A\in GL_{2}(\mf{f}), \,B\in [(\leftexp{2}{GSpin}_2\times GSpin_3)/\Delta GL_1](\mf{f})$ the action of $v[\alpha_i,\Theta^\ddag]$ on $\ms{M}$ is given by \bdm (A,B)\mapsto (w_0(\leftexp{t}A^{-1})w_0^{-1},(\mr{det}A)B).\edm

By [Bu, 4.1.1],  if $R'(A)=R_\sigma(\leftexp{t}{A}^{-1})$ then $R'\cong R_\sigma^\vee.$
For $i=0,2m$ if $\Theta=\Theta^+$, $i=2^-,4^-$ if $\Theta=\Theta^-,$ $i=1^\dag,5^\dag$ if $\Theta=\Theta^\dag$, and $i=0^\ddag,2^\ddag$ if $\Theta=\Theta^\ddag,$
\bdm \rho(v[\alpha_i,\Theta]\cdot(A,B))=R_\sigma(\leftexp{t}A^{-1})\omega_{R_\pi}(\mr{det}A)R_\pi(B)=R_\sigma^\vee(A)\omega_{R_\pi}(\mr{det}A)R_\pi(B).\edm
For $m=1$, as $R_\pi$ and  $R_\sigma$ are constructed using data from the same TRD parameter $\phi=\phi_1\oplus\phi_2,$ we have $\omega_{R_\pi}=\omega_{R_\sigma}.$ As $R_\sigma$ is a representation of $GL_2(\mf{f})$, by [Bu, 4.1.1], 
\bdm R_\sigma^\vee(\omega_{R_\sigma}\circ\mr{det})\cong R_\sigma.\edm

For $m=2$, both $R_\pi$ and $R_\sigma$ are constructed using data from the irreducible TRD parameter $\phi.$ 
For \bdm\phi(\mc{I}_t)=\langle s\rangle\subset GSp_4(\mb{C})\hookrightarrow GL_4(\mb{C}),\edm 
we have \bdm s=\mr{diag}(\tau,\tau^q,\tau^{q^2},\tau^{q^3}),\quad  \tau\tau^{q^2}=\tau^q\tau^{q^3}=:c.\edm
As in Section 4.5, for $R_\pi=R_{\ms{T}_\lambda,\chi_\lambda}$, the character $\chi_\lambda$ of $\ms{T}_\lambda^{F_\lambda}\subset GSpin_5(\mf{f})$ or $\leftexp{2}{GSpin}_4(\mf{f})$ is represented by $s.$ Similarly, for $R_\pi=R_{\ms{T}_\lambda,\chi_\lambda},$ the character $\chi_\lambda$ of $\ms{T}_\lambda^{F_\lambda}\subset GL_4(\mf{f})$ is represented by $s.$ We have
the character $\omega_\pi\circ\mr{det}$ of $GL_4(\mf{f})$ restricted to $\ms{T}_\lambda^{F_\lambda}$ is represented by the element \bdm\mr{diag}(c,c,c,c).\edm
Then, as $R_\pi^\vee=R_{\ms{T}_\lambda,\chi_\lambda^{-1}}$, by [Ca, 7.2.8] the representation
$R_\sigma^\vee\otimes (\omega_\pi\circ \mr{det})$ of $GL_4(\mf{f})$ is represented by the element \bdm s^{-1}c.\edm
We have \bdm s_2s_3s_1s_2\cdot s^{-1}c=s\edm for $s_2s_3s_1s_2\in N_{GL_4(\mb{C})}(\hat{T})/\hat{T}$ where $s_2s_3s_1s_2$ commutes with the Coxeter element $\hat{w}=s_1s_2s_3\in N_{GL_4(\mb{C})}(\hat{T})/\hat{T}$.
Then by [Ca, 7.3.4], $R_\sigma^\vee(\omega_{R_\pi}\circ\mr{det})\cong R_\sigma.$

Therefore, for $i=0,2m$ if $\Theta=\Theta^+$, $i=2^-,4^-$ if $\Theta=\Theta^-,$ $i=1^\dag,5^\dag$ if $\Theta=\Theta^\dag$, and $i=0^\ddag,2^\ddag$ if $\Theta=\Theta^\ddag,$ 
\bdm v[\alpha_i,\Theta]\cdot \rho\cong\rho.\edm

(ii) We will now show that the action of $\nu$ on $\ms{M}=GSpin_4(\mf{f})\times GL_2(\mf{f})$ does not preserve $\rho=R_\pi^-\boxtimes R_\sigma^-.$ 
We have $\nu$ acts on the simple roots in $\Theta^-$ by, \bdm \nu\cdot \alpha_{0^-}=\alpha_{1^-}, \nu\cdot\alpha_{1^-}=\alpha_{0^-}, \nu\cdot\alpha_{2^-}=\alpha_{2^-}, \nu\cdot\alpha_{3^-}=\alpha_{3^-}, \nu\cdot\alpha_{4^-}=\alpha_{4^-}.\edm
Note that $\nu$ fixes $R_\sigma^-$, but it suffices to show $\nu\cdot R_\pi^-\ncong R_\pi^-,$ which we now do. 

Up to conjugation, there is only one minisotropic torus in $GSpin_4(\mf{f}).$ So there exists $g\in GSpin_4(\mf{f})$ such that $\nu\cdot \ms{T}_\lambda=g\ms{T}_\lambda g^{-1}.$
Then by [D] there exists a lift $\dot{g}$ of $g$ to $\mc{P}$ such that $\nu\cdot T_\lambda=\dot{g} T_\lambda\dot{g}^{-1}.$
By replacing $\nu$ by $\mr{Ad}(\dot{g}^{-1})\circ\nu$ we can assume $\nu$ fixes $T_\lambda.$ Therefore 
\bdm \nu\in N_{GSpin_5(K)}(T_\lambda)/T_\lambda\edm
gives a non-trivial action on $T_\lambda.$ In fact, since $\nu$ induces a non-trivial diagram automorphism for the diagram of $GSpin_4$, 
\bdm \nu \in N_{GSpin_5(K)}(T_\lambda)/T_\lambda\setminus N_{GSpin_4(K)}(T_\lambda)/T_\lambda.\edm
Then, since for $R_\pi^-=R_{\ms{T}_\lambda,\chi_\lambda}$ the character $\chi_\lambda$ is in general position,
by [Ca, 7.3.4]
\bdm \nu\cdot\rho\ncong\rho.\edm 
\end{proof}

Define \bdm W(\Theta,\rho)=\{w\in S_\Theta\vert w\rho\simeq\rho\}.\edm
The elements of $W(\Theta,\rho)$ parametrize a basis for $\mc{H}(G,\rho)$, with relations given by Theorem (7.12) of [Mo2].

\begin{corollary} \bdm \mc{H}(G,\rho)=\langle T_a,T_b,T_c\rangle,\edm $\{a,b\}=\{0^+,2m^+\},\{2^-,4^-\},\{1^\dag,5^\dag\}$ or  $\{0^\ddag,2^\ddag\}$, subject for some $p_i$ to the relations
\begin{itemize}
\item[\textup{(}i\textup{)}] $T_c\, T_i=T_i\, T_c,$
\item[\textup{(}ii\textup{)}] $T_{i}^{2}=(p_i-1)T_i+p_i,$ 
\end{itemize}
where $i\in\{a,b\}.$
\end{corollary}

\begin{proof} We have shown \bdm W(\Theta,\rho)=\langle v[\alpha_{a},\Theta],v[\alpha_{b},\Theta], T(e_0^*) \rangle,\edm for 
$\{ a,b\}=\{0^+,2m^+\},\{2^-,4^-\},\{1^\dag,5^\dag\}$ or  $\{0^\ddag,2^\ddag\}.$
By [Mo2, Thm 7.12], $\mc{H}(G,\rho)$ is generated by three elements $T_a,T_b,T_c$ subject to the given relations. Note that the cocycle $\mu$ is trivial in all cases, and in particular $\mu$ is trivial when restricted to $\langle T(e_0^*) \rangle$ as $T(e_0^*)$ is a translation. Also, $T_c\, T_i=T_i \,T_c$ for all $i$ as $T(e_0^*)v[\alpha_i,\Theta]=v[\alpha_i,\Theta]T(e_0^*)$ (in $W$) for all $i$. This is as $wT(e_0^*)w^{-1}(x)=T(w(e_0^*))(x)=T(e_0^*)(x)$ for $x\in\mc{A}, w\in W_o,$ so all simple reflections commute with $T(e_0^*)$ in $W.$
\end{proof}

\subsection{A theorem of Lusztig}

In this section, we will compute the parameters $p_i$ in Corollary 6.3.
The support of $T_i$ is $\mc{P}\dot{w}_i\mc{P}$ where $\dot{w}_i\in N_G(T)$ projects to 
\begin{align*} & v[\alpha_i,\Theta^+]\in S_{\Theta^+}, \, i=0,2m, \, \mr{if} \,\, \mc{P}=\mc{P}^{+}, \quad
v[\alpha_i,\Theta^-]\in S_{\Theta^-}, \,  i=2^-,4^-, \, \mr{if} \,\, \mc{P}=\mc{P}^-, \\
& v[\alpha_i,\Theta^\dag]\in S_{\Theta^\dag}, \, i=1^\dag,5^\dag, \,\mr{if} \,\, \mc{P}=\mc{P}^{\dag},\quad\,\,\,\,\, v[\alpha_i,\Theta^\ddag]\in S_{\Theta^\ddag}, \, i=0^\ddag,2^\ddag, \, \mr{if} \,\, \mc{P}=\mc{P}^{\ddag}.\end{align*} 
The elements $\dot{w}_i$ lie in $\mc{K}_i.$ So $T_{i}$ is supported in $\mc{K}_{i},$ where $i$ is given as above.
Let $\ms{G}_i$ be the quotient of $\mc{K}_i$ by its pro-unipotent radical.
Let $\ms{P}_i=\ms{M}_i\ms{N}_i$ be the image of $\mc{P}$  in $\ms{G}_i$. 
Recall that $\ms{M}_i\cong \ms{G}_\lambda^{F_\lambda}\times GL_{2m}(\mf{f}),$ so that $\rho$ is a representation of $\ms{M}_i.$
Consider $V_i=\mr{Ind}_{\ms{P}_i}^{\ms{G}_i}(\rho).$
The algebra of right 
$\ms{G}_i$-endomorphisms of $V_i$,
\bdm\mr{End}_{\ms{G}_i}(V_i)=\mc{H}(\ms{G}_i,\rho),\edm
is isomorphic to the Hecke algebra
\bdm \mc{H}(\ms{G}_i,\rho)=\{f:\ms{G}_i\longrightarrow \mr{End}_{\mb{C}}(W^\vee)\vert f(p_1gp_2)=\check{\rho}(p_1)f(g)\check{\rho}(p_2),\,\,p_1,p_2\in\ms{P}_i,g\in\ms{G}_i\},\edm
where $(\check{\rho},W^\vee)$ is the contragredient representation of $(\rho,W).$
The finite Hecke algebra $\mc{H}(\ms{G}_i,\rho)$ can be canonically identified with a subalgebra of $\mc{H}(G,\rho)$. This is because $\mr{Ind}_{\mc{P}}^{\mc{K}_i}(\rho)\cong \mr{infl}_{\ms{G}_i}^{\mc{K}_i}(\mr{Ind}_{\ms{P}_i}^{\ms{G}_i}(\rho))$
and $\mc{H}(\mc{K}_i,\rho)$ can be canonically identified with a subalgebra of $\mc{H}(G,\rho)$.
By [Mo2, 6.5] $\mc{H}(\ms{G}_i,\rho)$ is two dimensional. It is generated by a function $T_e$ supported on  $\ms{P}_i$ and a function $T_i$ supported on $\ms{P}_iw_i\ms{P}_i$. Here, $w_i$ is the image of $\dot{w}_i$ in $\ms{G}_i.$ We have $w_i\in N_{\ms{G}_i}(\ms{M}_i)\cap N_{\ms{G}_i}(\ms{T}),$ and $w_i\cdot\rho\cong\rho.$
The function $T_i$ satisfies
\[T_i^2=(p_i-1)T_i+p_i.\]
Since the endomorphism algebra of $V_i$ is two dimensional it has two irreducible summands:
\bdm\mr{Ind}_{\ms{P}_i}^{\ms{G}_i}(\rho)=\rho_1\oplus\rho_2.\edm
By [HL], the parameter $p_i$ is the quotient of the degrees of the two irreducible summands.

We will use a theorem of Lusztig, [Lu, Thm 8.6], to compute the parameters $p_i.$ 


\subsection{Identification of $\ms{G}_i(\bar{\mf{f}})$}

We now describe $\ms{G}_i(\bar{\mf{f}})$ in the cases of concern.
With the identification of the character group of $T$ with the character group of $\ms{T}$, we have the character and cocharacter lattices for $\ms{T}$ are those for $T$,
\bdm X=\mb{Z}e_0\oplus\mb{Z}e_1\oplus\dots\oplus\mb{Z}e_n, \quad
X^\vee=\mb{Z}e_0^*\oplus\mb{Z}e_1^*\oplus\dots\oplus\mb{Z}e_n^*,\edm where $n=2m+2.$
When the simple roots for $T$ are given by $\Delta^+$, the fundamental chamber $C^+$ in the apartment $\mc{A}=X^\vee\otimes\mb{R}$ for $G=GSpin_{4m+5}$ is defined by the inequalities
\bdm 1-e_2>e_1>e_2>\dots>e_n>0.\edm
When the simple roots for $T$ are given by $\Delta^-$, the fundamental chamber $C^-$ in the apartment $\mc{A}$ for $G=GSpin_{9}$ is defined by the inequalities
\bdm 1-e_4>e_3>e_4>e_1>e_2>0.\edm
When the simple roots for $T$ are given by \bdm \Delta^\dag=\{\beta_{1^\dag}=e_6-e_5,\beta_{2^\dag}=e_5-e_1, \beta_{3^\dag}=e_1-e_2, \beta_{4^\dag}=e_2-e_3, \beta_{5^\dag}=e_3-e_4, \beta_{6^\dag}=e_4\}, \edm 
the fundamental chamber $C^\dag$ in the apartment $\mc{A}$ for $G=GSpin_{13}$ is defined by the inequalities
\bdm 1-e_5>e_6>e_5>e_1>e_2>e_3>e_4>0. \edm
When the simple roots for $T$ are given by \bdm \Delta^\ddag=\{\beta_{1^\ddag}=e_4-e_1,\beta_{2^\ddag}=e_1-e_2, \beta_{3^\ddag}=e_2-e_3, \beta_{4^\ddag}=e_3\}, \edm 
the fundamental chamber $C^\ddag$ in the apartment $\mc{A}$ for $G=GSpin_9$ is defined by the inequalities
\bdm 1-e_1>e_4>e_1>e_2>e_3>0. \edm
Let $x_i$ be the vertex of the fundamental chamber $C$ corresponding to the root $\alpha_i,$ (or $\beta_i$) in the local Dynkin diagram $\Pi$ for $G$.
The positive roots of $\ms{G}_i$ are
\bdm\Phi^{+}_{x_i}=\{a \in\Phi^{+}:\langle a, x_i \rangle\in \mb{Z}\},\edm where $\Phi^+$ is the set of positive roots for $T$ given by the simple roots above in the various cases.
The coroot associated with a root $a\in\Phi^+_{x_i}$ is the same for $\ms{G}_i$ as for $G$.
We will also list $\ms{M}_i$ and the dual groups $\ms{M}_i^*, \ms{G}_i^*$ in each case.

When $\mc{P}=\mc{P}^+$, the reductive component of reduction mod $\mf{p}$ of $\mc{P}$ is
\bdm \ms{M}_{0^+}=\ms{M}_{2m^+}=GL_{2m}\times GSpin_5, \quad\quad
 \ms{M}_{0^+}^*=\ms{M}_{2m^+}^*=GL_{2m}\times GSp_4. \edm
The root $\alpha_{0^+}$ corresponds to the vertex $x_0=0$, so the
set of positive roots for $\ms{G}_{0^+}$ is
\bdm\Phi^+_{x_0}=\{e_i\pm e_j\,\,\vert\,1\le i<j\le n\}\cup\{e_i\,\,\vert\, 1\le i\le n\}.\edm
 Then \bdm \ms{G}_{0^+}=GSpin_{4m+5},\quad \quad\ms{G}_{0^+}^*=GSp_{4m+4}.\edm
The root $\alpha_{2m^+}$ corresponds to the vertex $x_{2m}=1/2(e_1^*+e_2^*+\dots+e_{2m}^*)$, so
\bdm \Phi^+_{x_{2m}}=\{e_i\pm e_j \vert 1\leq i< j\leq 2m\}\cup\{e_{2m+1}-e_{2m+2}, e_{2m+2}, e_{2m+1}, e_{2m+1}+e_{2m+2}\}.\edm
Therefore, \bdm \ms{G}_{2m^+}=(GSpin_{4m}\times GSpin_5)/\Delta GL_1,\quad \ms{G}_{2m^+}^*=(GSO_{4m}\times GSp_4)^\circ.\edm

When $\mc{P}=\mc{P}^-$,
\bdm \ms{M}_{2^-}=\ms{M}_{4^-}=GL_2\times GSpin_4, \quad\quad \ms{M}_{2^-}^*=\ms{M}_{4^-}^*=GL_2\times GSO_4. \edm
The root $\alpha_{2^-}$ corresponds to the vertex $x_{2^-}=1/2(e_3^*+e_4^*),$ so the set of positive roots for $\ms{G}_{2^-}$ is
\bdm \Phi^+_{x_{2^-}}=\{e_1-e_2, e_2, e_1, e_1+e_2\}\cup\{e_3-e_4, e_3+e_4\}.\edm
Then we have \bdm \ms{G}_{2^-}=(GSpin_5\times GSpin_4)/\Delta GL_1, \quad\quad
 \ms{G}_{2^-}^*=(GSp_4\times GSO_4)^\circ. \edm
The root $\alpha_{4^-}$ corresponds to the vertex $x_{4^-}=1/2(e_1^*+e_2^*+e_3^*+e_4^*),$ so
\bdm \Phi^+_{x_{4^-}}=\{ e_3\pm e_4, e_4\pm e_1, e_1\pm e_2\}.\edm
Therefore
\bdm \ms{G}_{4^-}=GSpin_8,\quad\quad
 \ms{G}_{4^-}^*=GSO_8.\edm
 
When $\mc{P}=\mc{P}^\dag$,
\bdm \ms{M}_{1^\dag}(\bar{\mf{f}})=\ms{M}_{5^\dag}(\bar{\mf{f}})=GL_4\times GSpin_4, \quad\quad \ms{M}_{1^\dag}^*(\bar{\mf{f}})=\ms{M}_{5^\dag}^*(\bar{\mf{f}})=GL_4\times GSO_4. \edm
The root $\beta_{2^\dag}$ corresponds to the vertex $x_{2^\dag}=1/2(e_5^*+e_6^*),$ so the set of positive roots for $\ms{G}_{1^\dag}(\bar{\mf{f}})$ is
\bdm \Phi^+_{x_{2^\dag}}=\{e_i\pm e_j \vert \,\,1\le i<j\le 4\}\cup\{e_1, e_2, e_3, e_4\}\cup\{e_6\pm e_5\}.\edm
Then we have \bdm \ms{G}_{1^\dag}(\bar{\mf{f}})=(GSpin_9\times GSpin_4)/\Delta GL_1, \quad\quad
 \ms{G}_{1^\dag}^*(\bar{\mf{f}})=(GSp_8\times GSO_4)^\circ. \edm
The root $\beta_{6^\dag}$ corresponds to the vertex $x_{6^\dag}=1/2(e_1^*+e_2^*+e_3^*+e_4^*+e_5^*+e_6^*).$
We have
\bdm \ms{G}_{5^\dag}(\bar{\mf{f}})=GSpin_{12},\quad\quad
 \ms{G}_{5^\dag}^*(\bar{\mf{f}})=GSO_{12}.\edm

When $\mc{P}=\mc{P}^\ddag$,
\bdm \ms{M}_{0^\ddag}(\bar{\mf{f}})=\ms{M}_{2^\ddag}(\bar{\mf{f}})=GL_2\times (GSpin_2\times GSpin_3)/\Delta GL_1,\edm 
\bdm \ms{M}_{0^\ddag}^*(\bar{\mf{f}})=\ms{M}_{2^\ddag}^*(\bar{\mf{f}})=GL_2\times (GSO_2\times GSp_2)^\circ. \edm
The affine roots $1-e_4-e_1$ and $\beta_{1^\ddag}=e_4-e_1$ corresponds to the edge of $C^\ddag$ containing the vertices $0$ and $1/2(e_4^*),$ so the set of positive roots for $\ms{G}_{0^\ddag}(\bar{\mf{f}})$ is
\bdm\{e_1\pm e_2,e_2\pm e_3, e_1\pm e_3, e_1, e_2, e_3\}.\edm
Then we have \bdm \ms{G}_{0^\ddag}(\bar{\mf{f}})=(GSpin_7\times GSpin_2)/\Delta GL_1, \quad\quad
 \ms{G}_{0^\ddag}^*(\bar{\mf{f}})=(GSp_6\times GSO_2)^\circ. \edm
The root $\beta_{3^\ddag}$ corresponds to the vertex $x_{3^\ddag}=1/2(e_4^*+e_1^*+e_2^*),$ so
\bdm \Phi^+_{x_{3^\ddag}}=\{ e_4\pm e_1, e_1\pm e_2, e_4\pm e_2\}\cup\{e_3\}.\edm
Therefore
\bdm \ms{G}_{2^\ddag}(\bar{\mf{f}})=(GSpin_3\times GSpin_6)/\Delta GL_1,\quad\quad
 \ms{G}_{2^\ddag}^*(\bar{\mf{f}})=(GSp_2\times GSO_6)^\circ.\edm

To compute the parameters $p_i$, [Lu, 8.6] applies only when $\ms{G}_i$ has a connected center. To use this theorem in the cases where $\ms{G}_i$ does not have connected center, we will map $\ms{G}_i\hookrightarrow\ms{G}'_i$ where $\ms{G}'_i$ has connected center in all cases.
Let $\ms{M}'_i$ be the Levi component of the parabolic subgroup $\ms{P}'_i\subset\ms{G}'_i$ such that $\ms{M}_i\subset\ms{M}'_i$ and $\ms{P}_i\subset\ms{P}'_i$. 
We list the groups $\ms{M}'_i$ in the various cases in the following table. The groups $GSpin_{2n}^\sim$ are defined in the appendix, where their root datum are also given.

\renewcommand{\arraystretch}{1.6}

\begin{center}\begin{tabular}{| c | c | c |}
\hline
$i$ & $\ms{G}'_i(\bar{\mf{f}})$ & $\ms{M}'_i(\bar{\mf{f}})$ \\
\hline
$0^+$ & $GSpin_{4m+5}$ & $GL_{2m}\times GSpin_5$ \\
\hline
$2m^+$ & $(GSpin_{4m}^\sim \times GSpin_5)/\Delta GL_1$ & $GL_1\times GL_{2m}\times GSpin_5$ \\
\hline
$2^-$ & $(GSpin_5\times GSpin_4^\sim)/\Delta GL_1$ & $GL_2\times GSpin^\sim_4$ \\
\hline
$4^-$ & $GSpin_8^\sim$ & $GL_2\times GSpin^\sim_4$ \\
\hline
$1^\dag$ & $(GSpin_9\times GSpin_4^\sim)/\Delta GL_1$ & $GL_4\times GSpin^\sim_4$ \\
\hline
$5^\dag$ & $GSpin_{12}^\sim$ & $GL_4\times GSpin^\sim_4$ \\
\hline 
$0^\ddag$ & $(GSpin_7\times GSpin_2)/\Delta GL_1$ 
 & $GL_2\times (GSpin_2\times GSpin_3)/\Delta GL_1$   \\
\hline
$2^\ddag$ & $(GSpin_3\times GSpin_6^\sim)/\Delta GL_1$  
& $GL_1\times GL_2\times (GSpin_2\times GSpin_3)/\Delta
 GL_1$  \\
 \hline
\end{tabular}\end{center}


\subsection{Calculation of parameters $p_i$}



Let \bdm\ms{T}_i=\ms{T}_{GL_{2m}}\times \ms{T}_{\ms{G}_\lambda}\subset \ms{M}_i,\edm where $\ms{T}_{GL_{2m}}$ is the split maximal torus in $GL_{2m}(\bar{\mf{f}})$ and $\ms{T}_{\ms{G}_\lambda}$ is the split maximal torus in $\ms{G}_\lambda(\bar{\mf{f}}).$
Recall from Section 4 that there exists an element $\overline{p_\lambda}\in \ms{G}_\lambda(\bar{\mf{f}})$ such that 
$\ms{T}_\lambda^{F_\lambda}=\mr{Ad}(\overline{p_\lambda})\ms{T}^{F_w},$ where $w$ is defined in Section 4.5.
Also, from Section 5.1, we have an element $p_{GL_{2m}}\in GL_{2m}(\bar{\mf{f}})$ such that  $\ms{S}^{F}=\mr{Ad}(p_{GL_{2m}})\ms{T}^{F_w},$ where in this case $w$ is a Coxeter element.
Denote by \bdm\ms{S}_i^{F_\lambda}=\mr{Ad}(p_{GL_{2m}},\overline{p_{\lambda(\ms{G}_\lambda)}})\ms{T}_i^{F_w}\edm where
\bdm F_\lambda=F_{(GL_{2m})}\otimes F_{\lambda(\ms{G}_\lambda)},\quad F_w=F_{w(GL_{2m})}\otimes F_{w(\ms{G}_\lambda)}.\edm We will also denote by  $F_\lambda, F_w$ the twisted Frobenius which acts on $\ms{G}'_i$ such that the restriction of to $\ms{G}_i$ is $F_\lambda, F_w.$ Also, $F^*_\lambda, F^*_w$ will denote the dual Frobenius which acts on $\ms{G}'^*_i.$

The character $\chi_\lambda=\eta_{(GL_{2m})}\otimes\chi_{\lambda(\ms{G}_\lambda)}$ of the $F_\lambda=F_{(GL_{2m})}\otimes F_{\lambda(\ms{G}_\lambda)}$
stable torus $\ms{S}_i$ can be viewed as a regular element in a dual $F_\lambda^*$ stable torus $\ms{S}_i^*$ of the dual group $\ms{G}_i^*$.
Let $t$ correspond to $\chi_\lambda=\eta_{(GL_{2m})}\otimes\chi_{\lambda(\ms{G}_\lambda)}.$ There is a surjective map of dual groups 
\bdm\psi:\ms{G}'^*_i\longrightarrow \ms{G}_i^*.\edm
Let $t'$ be an element of $\ms{G}'^*$ such that $\psi(t')=t,$ and let $\chi'_\lambda$ be the character given by $t'.$
Since $\chi'_\lambda$ is in general position, we have that $t'$ defines an irreducible representation $\rho'$ of $\ms{M}'_i(\mf{f})$. 



\begin{lemma} The restriction of $\rho'$ to $\ms{M}_i(\mf{f})$ is isomorphic to $\rho.$ \end{lemma}

\begin{proof} For $i=0^+,0^\ddag$, $\ms{M}'_i=\ms{M}_i$ so we will not consider these cases. 
Let $R_{\ms{S}'_i}(t')$ be the virtual character (up to a sign) of the representation of $\ms{M}'^{F_\lambda}_i$ given by $t'$. Then $\ms{S}_i^{F_\lambda}\subset \ms{S}'^{F_\lambda}_i$ where $t$ gives the character of $\chi_\lambda$ of $\ms{S}_i^{F_\lambda}$ and $t'$ gives a character $\chi'_\lambda$ of $\ms{S}'^{F_\lambda}_i$ such that 
$\chi'_\lambda(t)=\chi_\lambda(t)$ for $t\in\ms{S}_i^{F_\lambda}.$ Let $R\vert_{\ms{M}_i}$ be the restriction of $R_{\ms{S}'_i}(t')$ to $\ms{M}_i^{F_\lambda}.$
By the character formula [Ca, 7.2.8] we have
\bdm R\vert_{\ms{M}_i}(g)=R_{\ms{S}_i}(t)(g),\quad g\in\ms{M}_i,\edm
so they are the same virtual character. Since $\epsilon_{\ms{M}'_i}\epsilon_{\ms{S}'_i}=\epsilon_{\ms{M}_i}\epsilon_{\ms{S}_i}$ we have
 $\rho' \vert_{\ms{M}_i(\mf{f})}\cong \rho.$
\end{proof}

\begin{lemma} The homomorphism $\psi:\ms{G}'^*_i\longrightarrow \ms{G}_i^*$ central kernel. \end{lemma}

\begin{proof} For $i=0^+,0^\ddag$, $\ms{T'}_i^*=\ms{T}_i^*$, and for $i=2m^+, 2^\ddag$, the kernel of $\psi$ is $\{(g,1)\vert \,\,g\in GL_1\}$ which is in the center of  $GL_1\times\ms{M}_i^*.$ 
For the remaining cases it suffices to show the homomorphism $\psi\vert_{(GSpin_{2n}^\sim)^*}:\ms{T}'^*\subset(GSpin_{2n}^\sim)^*\longrightarrow \ms{T}^*\subset(GSpin_{2n})^*$ has kernel contained in the center of $\ms{G}'^*_i.$
The center of $(GSpin_{2n}^\sim)^*$,
\bdm \{E_{-1}(\mu)E_1(\nu)\dots E_n(\nu)E_0(\nu^2);\,\, \mu,\nu\in GL_1\},\edm
 is given by all elements of the split torus $\ms{T'}_i^*$ that belong to the kernel of all the simple roots. 
We have \bdm \psi\vert_{(GSpin_{2n}^\sim)^*}:\ms{T'}^*=\prod_{j=-1}^n E_j(\lambda_j)\longrightarrow \ms{T}^*=\prod_{j=0}^n E_j(\lambda_j),\quad \lambda_j\in GL_1,\edm is given by
\bdm E_{-1}(\lambda_{-1})E_0(\lambda_0)\dots E_n(\lambda_n)\quad\mapsto \quad E_0(\lambda_0)\dots E_n(\lambda_n),\edm which has kernel contained in the center of $\ms{G}'^*_i.$
\end{proof}

\begin{lemma} The group $C_{\ms{G}'^*_i(\bar{\mf{f}})}(t')$ is connected, reductive with root system
\begin{itemize}
\item[\textup{(}i\textup{)}] type $(A_1)^2$ and Weyl group $W_{A_1}^2$ if $i=2m^+\,(m=2), \, 1^\dag$;
\item[\textup{(}ii\textup{)}] type $(A_2)^2$ and Weyl group $W_{A_2}^2$ if $i=0^+\,(m=2), \,\,\,\,\,\, 5^\dag;$
\item[\textup{(}iii\textup{)}] type $A_1$ and Weyl group $W_{A_1}$ if  \,\,\,\,\, $i=2m^+\,(m=1), \, 2^-, \, 2^\ddag\,(case \,1), \, 0^\ddag\,(case\, 2)$;
 \item[\textup{(}iv\textup{)}] type $A_2$ and Weyl group $W_{A_2}$ if \,\,\,\,\, $i=0^+\,(m=1), \,\,\,\,\,\, 4^-, \, 0^\ddag\,(case \,1), \, 2^\ddag\,(case\, 2)$.
\end{itemize}
In each case there is one orbit on the simple roots for $W_{\ms{G}'^*_i},$ the Weyl group of $C_{\ms{G}'^*_i(\bar{\mf{f}})}(t'),$ under the action of the dual Frobenius $F^*_\lambda$ on $\ms{G}'^*_i.$ 
In addition, $C_{\ms{M}'^*_i(\bar{\mf{f}})}(t')=\ms{S}'^*_i.$
\end{lemma}

\begin{proof} Over $\bar{\mf{f}}$, there exists $x\in \ms{M}_i^*$ such that $x\ms{S}_i^*x'^{-1}=\ms{T}_i^*$. 
Again over $\bar{\mf{f}}$, there exists $x'\in \ms{M}'^*_i$ such that $\psi(x')=x$ and $x'\ms{S}'^*_ix'^{-1}=\ms{T}'^*_i$. 
Let $s=x^{-1}t x$ and $s'=x'^{-1}t'x',$ where $\psi(s')=s.$ We have an
isomorphism of reductive groups
\bdm \varphi : C_{\ms{G}'^*_i}(s')\rightarrow C_{\ms{G}'^*_i}(t'), \quad \varphi(z)=x'zx'^{-1}.\edm
So it suffices to compute $C_{\ms{G}'^*_i}(s')$ with its $F^*_w$ structure, which we do in the following.

As the center of $\ms{G}'_i$ is connected, centralizers of semisimple elements in $\ms{G}'^*_i$ are connected. Therefore $C_{\ms{G}'^*_i}(s')$ is generated by $\ms{T}'^*_i$ and the root groups $U_a$ such that $a(s')=1.$
Since the map \bdm\psi:\ms{G}'^*_i\rightarrow \ms{G}^*_i\edm has central kernel,
it suffices to compute the root groups $U_a$ such that $a(s)=1.$

The elements $s$ are constructed from the elements, also called $s$, listed in the tables in Section 4.5 and Section 5.1. 
In the case $\rho=R_\sigma^\ddag\otimes R_\pi^\ddag,$ we have
$s=(\mr{diag}(\tau_i, \tau_i^q),\mr{diag}(\tau_i, \tau_j, \tau_j^q, \tau_i^q))$ or $s=(\mr{diag}(\tau_i, \tau_i^q),\mr{diag}(\tau_j, \tau_i, \tau_i^q, \tau_j^q))$, $i\ne j \in \{1, 2\},$  is an element of $GL_2\times (GSO_2(\bar{\mf{f}})\times GSp_2(\bar{\mf{f}}))^\circ.$ Let case 1 and case 2 refer to the situations where
\[\textrm{case 1}\longleftrightarrow \mr{diag}(\tau_i,\tau_i^q)\in GSp_2(\bar{\mf{f}}),\quad \textrm{case 2}\longleftrightarrow \mr{diag}(\tau_i,\tau_i^q)\in GSO_2(\bar{\mf{f}}).\]
Let $I$ be the root system of $W_{\ms{G}'^*_i}$, the Weyl group of $C_{\ms{G}'^*_i}(s').$ Then in the following table we display the different possibilities.


\begin{center}\begin{tabular}{| c | c | c | c |}
\hline
$i$ & $\ms{G}_i^*(\bar{\mf{f}})$ & $I$ & $W_{\ms{G}'^*_i}$ \\
\hline
$4^+$ & $(GSO_{8}\times GSp_4)^\circ$ & $\{\pm(e_1^*+e_3^*-e_0^*), \pm(e_2^*+e_4^*-e_0^*)\}$ & $W_{A_1}^2$ \\
\hline
$0^+$ & $GSp_{12}$ & $\langle e_1^*-e_5^*, e_3^*+e_5^*-e_0^*, e_2^*-e_6^*, e_4^*+e_6^*-e_0^* \rangle$ & $W_{A_2}^2$ \\
\hline
$2^+$ & $(GSO_{4}\times GSp_4)^\circ$ & $\{\pm(e_1^*+e_2^*-e_0^*)\}$ & $W_{A_1}$ \\
\hline
$0^+$ & $GSp_{8}$ & $\langle e_1^*-e_3^*, e_2^*+e_3^*-e_0^* \rangle$ & $W_{A_2}$ \\
\hline
$2^-$ & $(GSp_4\times GSO_4)^\circ$ & $\{\pm(e_1^*+e_2^*-e_0^*)\}$ & $W_{A_1}$ \\
\hline
$4^-$ & $GSO_8$ & $\langle e_1^*-e_3^*, e_2^*+e_3^*-e_0^* \rangle$ & $W_{A_2}$ \\
\hline
$1^\dag$ & $(GSp_8\times GSO_4)^\circ$ & $\{\pm(e_1^*+e_3^*-e_0^*), \pm(e_2^*+e_4^*-e_0^*)\}$ & $W_{A_1}^2$ \\
\hline
$5^\dag$ & $GSO_{12}$ & $\langle e_1^*-e_5^*, e_3^*+e_5^*-e_0^*, e_2^*-e_6^*, e_4^*+e_6^*-e_0^* \rangle$ & $W_{A_2}^2$ \\
\hline
$0^\ddag$ & $(GSp_6\times GSO_2)^\circ$ & case 1, \quad$\langle e_1^*-e_3^*, e_2^*+e_3^*-e_0^* \rangle$ & $W_{A_2}$ \\
\hline
$2^\ddag$ & $(GSp_2\times GSO_6)^\circ$ & case 1, \quad\quad $\{\pm(e_1^*+e_2^*-e_0^*)\}$ & $W_{A_1}$ \\
\hline
$0^\ddag$ & $(GSp_6\times GSO_2)^\circ$ & case 2, \quad\quad $\{\pm(e_1^*+e_2^*-e_0^*)\}$ & $W_{A_1}$ \\
\hline
$2^\ddag$ & $(GSp_2\times GSO_6)^\circ$ & case 2,  \quad$\langle e_1^*-e_4^*, e_2^*+e_4^*-e_0^* \rangle$ & $W_{A_2}$ \\
\hline
\end{tabular}\end{center}

\renewcommand{\arraystretch}{1}

Note that in each case $F^*_w$ acts transitively on the simple roots in $I$. For $i=4^+$, $F^*_w$ interchanges $e_1^*+e_3^*-e_0^*$ and $e_2^*+e_4^*-e_0^*$.
For $i=0^+$, 
\bdm F^*_w(e_1^*-e_5^*)=e^*_4+e_6^*-e_0^*,\quad F^*_w(e_3^*+e_5^*-e_0^*)=e^*_2-e_6^*,\edm 
\bdm F^*_w(e_2^*-e_6^*)=e^*_1+e_5^*, \quad F^*_w(e_4^*+e_6^*-e_0^*)=e^*_3+e_5^*-e_0^*.\edm
For $i=2^+$, $F^*_w$ fixes $e_1^*+e_2^*-e_0^*.$ For $i=0^+,$ $F_w^*$ interchanges $e_1^*-e_3^*$ and $e_2^*+e_3^*-e_0^*.$
The other cases are similar.

As $\chi_\lambda$ is in general position, $s$ is not fixed by any element of the Weyl group $(N_{\ms{M}_i^*}(\ms{T}_i^*)/\ms{T}_i^*)^{F^*_w}$ so $C_{\ms{M}_i^*}(s)=\ms{T}_i^*.$ Since $\psi:\ms{T}'^*_i\longrightarrow \ms{T}^*_i$ has central kernel, we have that
$C_{\ms{M}'^*_i}(s')=\ms{T}'^*_i$ which implies \bdm C_{\ms{M}'^*_i}(t')=\ms{S}'^*_i.\edm
\end{proof}

\begin{lemma} We have $p_i'=p_i.$ \end{lemma}

\begin{proof} By [Lu, Thm 8.6], \bdm \mr{End}_{\ms{G}'_i(\mf{f})}(\mr{Ind}_{\ms{P}'_i(\mf{f})}^{\ms{G}'_i(\mf{f})}\rho')=\langle T_e, T_i\rangle\edm  as an algebra, where $T_e$ is supported on $\ms{P}'_i$ and $T_i$ satisfies the relation
$T_i^2=(p'_i-1)T_i+p'_i.$ Here, $T_i$ corresponds to the unique $F^*_w$-orbit on $I.$
Therefore, since $\mr{End}_{\ms{G}'_i}(\mr{Ind}_{\ms{P}'_i}^{\ms{G}'_i}\rho')$ has dimension two, we have that
\bdm \mr{Ind}_{\ms{P}'_i}^{\ms{G}'_i}\rho'=\rho'_1\oplus\rho'_2,\edm where $\rho'_1$ and $\rho'_2$ are distinct irreducible representations of $\ms{G}'_i$.
As $\rho'\vert_{\ms{M}_i}=\rho,$ by Mackey's induction restriction theorem,
\bdm \mr{Ind}_{\ms{P}'_i}^{\ms{G}'_i}(\rho')\vert_{\ms{G}_i}\cong \mr{Ind}_{\ms{P}_i}^{\ms{G}_i}(\rho)\implies (\rho'_1\oplus\rho'_2)\vert_{\ms{G}_i}=\rho_1\oplus\rho_2.\edm
Therefore the quotient of the degrees of $\rho'_1$ and $\rho'_2$ is equal to the quotient of the degrees of $\rho_1$ and $\rho_2$, so $p'_i=p_i.$
\end{proof}

\begin{lemma} We have:
\begin{itemize}
\item[\textup{(}i\textup{)}] $p_i=q^2$ \, if  \,\, $i=2m^+\,(m=2), \, 1^\dag$;
\item[\textup{(}ii\textup{)}] $p_i=q^6$ \, if \,\, $i=0^+\,(m=2), \,\,\,\,\,\, 5^\dag;$
\item[\textup{(}iii\textup{)}] $p_i=q$ \,\,\, if  \,\, $i=2m^+\,(m=1), \; 2^-, \,\, 2^\ddag\,(case \,1), \,\, 0^\ddag\,(case \,2)$;
 \item[\textup{(}iv\textup{)}] $p_i=q^3$ \, if \,\, $i=0^+\,(m=1), \,\,\,\,\,\, 4^-, \,\, 0^\ddag\,(case \,1), \,\, 2^\ddag\,(case \,2)$.
\end{itemize}
\end{lemma}

\begin{proof}
Apply formula 8.2.3 of [Lu] to the pair $(W_{A_1}^n, 1)$. This calculation is done in [KM, Appendix], and we have
$p'_i=q^n.$
Apply formula 8.2.3 of [Lu] to the pair $(W_{A_2}^n, 1)$. This calculation is done in [Sa, \S 5] and we have
$p'_i=q^{3n}.$
From the previous lemma, $p'_i=p_i$ in each case.
\end{proof}

\begin{corollary} The parameters $p_i$ of $\mc{H}(G,\rho)$ are unequal.
\end{corollary}

\section{Reducibility of generalized principal series}

\subsection{A result of Matsumoto}

Let $(W,S)$ be a Coxeter system of type $\tilde{A}_1,$ 
so that $W$ is generated by $S=\{s_1,s_2\}$ where $s_1^2=s_2^2=1$ and $s_1s_2$ has infinite order.
For $k$ a commutative ring and $q$ a real, positive valued, quasi-multiplicative function on $W$, we denote by $k(W,q)$ the Hecke algebra of type $\tilde{A}_1$ associated to $q$. 

Denote by $q_1=q(s_1), \,q_2=q(s_2),$ and assume  $q_2\ge q_1\ge 1.$
Matsumoto [Ma] defines representations $\pi_\xi$ of $k(W,q)$, indexed by $\xi\in\mb{C}^\times$, such that all irreducible finite dimensional representations of $k(W,q)$ 
and all irreducible unitary representations of $k(W,q)$ occur in the composition series of such representations. We have:
\begin{itemize}
\item[(i)] For $\vert\xi\vert=1$, the representations $\pi_\xi$ are irreducible and unitary. 
\item[(ii)] For $\xi\in \mb{R}$ such that $1<\xi \le \sqrt{q_1q_2}, -\sqrt{q_2/q_1}\le \xi<-1,$ we have $\pi_\xi$ is unitary. For $\xi\ne\sqrt{q_1q_2}, \,-\sqrt{q_2/q_1},$ $\pi_\xi$ is irreducible.
\item[(iii)] For $\xi=\sqrt{q_1q_2}, \,-\sqrt{q_2/q_1},$ the composition series of $\pi_\xi$ is of length two. 
\end{itemize}
Let  $\chi_{\xi'}$ denote the irreducible subrepresentation of $\pi_{\sqrt{q_1q_2}}$ and $\chi_{\xi''}$ the irreducible subrepresentation of $\pi_{-\sqrt{q_2/q_1}}.$
Let $S^1$ be the multiplicative group of complex numbers of modulus 1, and $d\xi$ be the Haar measure on $S^1$ such that $\int_{S^1} d\xi=1$. 
Matsumoto then gives the Plancherel formula for $k(W,q).$ 
Let $L^1(W,q)$ be the Banach space of $L^1$ integrable functions $f\in k(W,q)$ with respect to a fixed Haar measure on $W$. 
There is a meromorphic function $c_1$ on $\mb{C}^\times$ such that
for all $f\in L^1(W,q)$,
\begin{align*} f(e)= &
\frac{1}{2}\int_{S^1} \mr{Tr}(\pi_\xi(f))\vert c_1(\xi)\vert^{-2}\, d\xi \\
& + \frac{1-q_1^{-1}q_2^{-1}}{(1+q_1^{-1})(1+q_2^{-1})}\mr{Tr}(\chi_{\xi'}(f)) + \frac{1-q_1q_2^{-1}}{(1+q_1)(1+q_2^{-1})}\mr{Tr}(\chi_{\xi''}(f)).\end{align*}

We have
\begin{theorem} \textup{(}Matsumoto\textup{)} A Hecke algebra of type $\tilde{A}_1$ has two complementary series if and only if $q_1\ne q_2$. \end{theorem}

\subsection{Plancherel measures}

To proceed we will need to use information about $\mu(s,\pi\boxtimes\sigma),$ the Plancherel measure associated to the generalized principal series $I(s,\pi\boxtimes\sigma).$ 
If $P=M\cdot N$ is the maximal parabolic of $GSpin_{4m+5}$ or $GSpin_{2m+4,2m+1}$ defined in Section 5.2, let $\bar{P}=M\cdot\bar{N}$ be the opposite parabolic, and $I_{\bar{P}}(s,\pi\boxtimes\sigma)$ the corresponding generalized principal series representation. There is a local intertwining operator
\bdm A(s,\pi\boxtimes\sigma,N,\bar{N}):I(s,\pi\boxtimes\sigma)\longrightarrow I_{\bar{P}}(s,\pi\boxtimes\sigma).\edm
The composite $A(s,\pi\boxtimes\sigma,\bar{N},N)\circ A(s,\pi\boxtimes\sigma,N,\bar{N})$ is a scalar operator on $I(s,\pi\boxtimes\sigma)$ and the Plancherel measure is the meromorphic function defined by 
\begin{displaymath} \mu(s,\pi\boxtimes\sigma)^{-1}=A(s,\pi\boxtimes\sigma,\bar{N},N)\circ A(s,\pi\boxtimes\sigma,N,\bar{N}).\end{displaymath}

From [Sil, 5.3-5.4] we have 
\begin{prop}\label{2} \textup{(}Harish-Chandra\textup{)} If $\pi\boxtimes\sigma$ is a unitary supercuspidal representation of $GSpin_5(k)\times GL_r(k)$ or $GSpin_{4,1}(k)\times GL_r(k),$ then for $s$ varying over the real numbers:
\begin{enumerate}
\item[\textup{(}i\textup{)}] If $\sigma\ncong\sigma^\vee(\omega_{\pi}\circ\mr{det})$, then $\mu(0,\pi\boxtimes\sigma)\ne 0$ and $I(s,\pi\boxtimes\sigma)$ is irreducible for all $s\in\mb{R}.$
In this case only $I(0,\pi\boxtimes\sigma)$ is unitary.
\item[\textup{(}ii\textup{)}] If $\sigma\cong\sigma^\vee(\omega_{\pi}\circ\mr{det})$, then there is a unique real $s_0\ge 0$ such that $I(s_0,\pi\boxtimes\sigma)$ is reducible. 
Moreover, $s_0> 0$ if and only if $
\mu(0,\pi\boxtimes\sigma)=0$, in which case $s_0$ is the unique pole of $\mu(s,\pi\boxtimes\sigma)$ on the positive real axis.
\begin{enumerate}
\item[\textup{(}a\textup{)}] When $s_0=0,$ for all $s\ne 0$ we have $I(s,\pi\boxtimes\sigma)$ is irreducible and non-unitary. The representation $I(s_0,\pi\boxtimes\sigma)$ is of length $2,$ with irreducible subquotients tempered representations.
\item[\textup{(}b\textup{)}] When $s_0>0,$ we have $I(s,\pi\boxtimes\sigma)$ is only reducible for $s=\pm s_0.$ For $\vert s\vert<\vert s_0\vert,$ $I(s,\pi\boxtimes\sigma)$ is unitary. The representation $I(s_0,\pi\boxtimes\sigma)$ is of length $2,$ with unique irreducible submodule a discrete series representation.
\end{enumerate}
\end{enumerate}\end{prop}


\begin{lemma} 
Let $\pi\boxtimes\sigma$ be a unitary supercuspidal representation where $\pi$ is a supercuspidal representation of $GSpin_{5}(k)$ or $GSpin_{4,1}(k)$ corresponding to a tame regular discrete series $L$-parameter $\phi=\phi_{1}\oplus\dots\oplus\phi_{r}$, $r=1,2,$ by the construction of DeBacker and Reeder given in Section 4.5.
Also $\sigma\cong\sigma_{\phi_i}$, where $1\le i\le r$, is the depth zero supercuspidal representation of $GL_{2m}(k)$ attached to the $L$-parameter $\phi_{i}$ via the local Langlands correspondence for $GL_{2m}.$
Then there are at most two twists $\vert\mr{det}\vert^{iv_j}$ of $\sigma$ such that 
\bdm \sigma \vert\mr{det}\vert^{iv_1}\ncong\sigma\vert\mr{det}\vert^{iv_2}\edm
and such that a representation 
\bdm \mr{Ind}_{P}^{G} \delta_{P}^{1/2} \pi\boxtimes\sigma\vert \mr{det} \vert^{u+iv_j}\edm could reduce for some $u>0$. 
\end{lemma}


\begin{proof} 
By Proposition 7.2, 
$I(s,\pi\boxtimes\sigma\vert\mr{det}\vert^{iv})$ can reduce for some real $s>0$ only if \bdm\sigma\vert\mr{det}\vert^{iv}\cong(\sigma\vert\mr{det}\vert^{iv})^\vee(\omega_{\pi}\circ\mr{det}).\edm  
Suppose $\sigma_0$ satisfies $\sigma_0\cong\sigma_0^\vee(\omega_{\pi}\circ\mr{det}).$
If $\sigma_0\vert\mr{det}\vert^{iw}$ satisfies
\bdm\sigma_0\vert\mr{det}\vert^{iw}\cong(\sigma_0\vert\mr{det}\vert^{iw})^\vee(\omega_{\pi}\circ\mr{det}),\edm then
\bdm \sigma_0\vert\mr{det}\vert^{2iw}\cong\sigma_0.\edm
Then for $g\in Z(GL_{2m}),$ $\vert\mr{det}(g)\vert^{2iw}=1$ so we must have 
\bdm w=\frac{-k\pi}{2m\mr{log}(q)}\edm for some integer $k$ such that $0\le k\le 2.$
By the local Langlands conjecture for $GL_{2m}$, if $\chi:k^\times\rightarrow \mb{C}^\times$, then 
\bdm \sigma_0(\chi\circ\mr{det})\cong\sigma_0\Longleftrightarrow \phi_{\sigma_0}\otimes\chi\cong \phi_{\sigma_0},\edm
where $\chi$ is viewed as a character of $W_k$ via local class field theory. We have $\phi_{\sigma_0}=\mr{Ind}_{W_{k_{2m}}}^{W_k} \eta.$ Then 
\bdm  \phi_{\sigma_0}\otimes\chi\cong \phi_{\sigma_0}\quad\Longleftrightarrow\quad \mr{Ind}_{W_{k_{2m}}}^{W_k}(\eta\cdot\chi\vert_{W_{k_{2m}}})\cong\mr{Ind}_{W_{k_{2m}}}^{W_k} \eta\edm
\bdm\Longleftrightarrow\quad \chi\vert_{W_{k_{2m}}}=1\quad\Longleftrightarrow \quad\chi^{2m}=1. \edm
Therefore if $\vert \cdot\vert^{iw}$ has order dividing $2m$, then $\sigma_0\vert\mr{det}\vert^{iw}\cong\sigma_0.$ Since $\vert\cdot\vert^{iw}$ must have order dividing $4m$, $\vert\cdot\vert^{iw}$ gives a nontrivial twist of $\sigma_0$ only if it has order $4m$. Any two such twists differ by a character of order dividing $2m$, so there is at most one twist $\vert\cdot\vert^{iw}$  of $\sigma_0$ such that  $\sigma_0\vert\mr{det}\vert^{iw}\cong(\sigma_0\vert\mr{det}\vert^{iw})^\vee(\omega_{\pi}\circ\mr{det}).$ 
\end{proof}

\subsection{Main theorem}


\begin{theorem} Let $\pi$ be a depth zero supercuspidal representation of $GSpin_{5}(k)$ or $GSpin_{4,1}(k)$ corresponding to a tame regular discrete series $L$-parameter $\phi=\phi_{1}\oplus\dots\oplus\phi_{r}$, $r=1,2,$ by the construction of DeBacker and Reeder given in Section 4.5.
Let $\sigma\cong\sigma_{\phi_i}$, where $1\le i\le r$, be the depth zero supercuspidal representation of $GL_{2m}(k)$ attached to the $L$-parameter $\phi_{i}$ via the local Langlands correspondence for $GL_{2m}.$
Then the generalized principal series $I(s,\pi\boxtimes\sigma)$ reduces at a unique $s_{0}>0.$ 
\end{theorem}

\begin{proof} Assume $\pi\boxtimes\sigma$ is unitary. 
Representations in the Bernstein component $\mf{R}^{[M,\pi\boxtimes\sigma]_G}(G)$ of $I(s,\pi\boxtimes\sigma)$ are parametrized by $\mc{H}(G,\rho)$-modules via the map 
\bdm M_\rho: \mf{R}^{[M,\pi\boxtimes\sigma]_G}(G)\rightarrow\mc{H}(G,\rho)-\mr{Mod},\quad (\kappa,V)\mapsto V_\rho.\edm
For an irreducible representation $\kappa=\mr{Ind}_{P}^{G} \delta_{P}^{1/2} \pi\vert\mr{sim}\vert^{t}\boxtimes\sigma\vert \mr{det} \vert^{s}$ in  $\mf{R}^{[M,\pi\boxtimes\sigma]_G}(G)$, the function 
\bdm T_c\in\mc{H}(G,\rho)\edm acts on $M_\rho(\kappa)$ by the scalar 
\bdm \mr{meas}(\mc{P})\,\omega_\pi(n)\vert\mr{sim}(n)\vert^t\edm where $n=e_0^*(\varpi^{-1}).$ 
Therefore irreducible representations in the Bernstein component of $I(s,\pi\boxtimes\sigma)$ which are a composition factor of some 
\bdm \mr{Ind}_{P}^{G} \delta_{P}^{1/2} \pi\vert\mr{sim}\vert^t \boxtimes\sigma\vert \mr{det} \vert^{u+iv_j},\quad t=0,\edm
 are parametrized by simple modules of a Hecke algebra of type $\tilde{A}_1.$ This Hecke algebra of type $\tilde{A}_1$ has unequal parameters by Corollary 6.9 and therefore has to have two complementary series by Theorem 7.1.
By Lemma 7.3 (up to isomorphism) there are only two possible $v_j$ such that the representation
$\mr{Ind}_{P}^{G} \delta_{P}^{1/2} \pi\boxtimes\sigma\vert \mr{det} \vert^{u+iv_j}$
could reduce for some $u>0.$
In the course of the proof of Lemma 6.2 we showed that for $\rho=R_\pi\boxtimes R_\sigma$, 
\bdm R_\sigma\cong R_\sigma^\vee(\omega_{R_\pi}\circ\mr{det}).\edm
Then since \bdm \sigma=\mr{Ind}(\chi_\lambda\otimes R_\sigma),\edm we have
\bdm \sigma\cong\sigma^\vee(\omega_{\pi}\circ\mr{det}).\edm
Therefore, by Proposition 7.2, $I(s,\pi\boxtimes\sigma)$ could reduce for some real $s>0$.
Since there must be two complementary series, $I(s,\pi\boxtimes\sigma)$ does reduce for a unique real $s_0>0.$
\end{proof}

\section{The $L$-packets agree}

When $\pi\boxtimes\sigma$ is irreducible and generic as a representation of $M$, we can apply Shahidi's theory of $L$-functions.
Here $\pi$ is a generic representation of $GSpin_5(k).$ 
The dual parabolic subgroup of $P$ is $P^{\vee}=M^{\vee}\cdot N^{\vee}\subset GSpin_{4m+5}^{\vee}=GSp_{4m+4}(\mathbb{C}),$ where \begin{displaymath} M^{\vee}=GSp_{4}(\mathbb{C})\times GL_{2m}(\mathbb{C}).\end{displaymath} Under the adjoint action of $M^{\vee},$ $\mathfrak{n}^{\vee}=Lie(N^{\vee})$ decomposes as $r_{1}\oplus r_{2},$ where each $r_{i}$ is a maximal isotypic component for the action of the central torus in $M^{\vee}.$ 
In this case
\[ r_{1}=std^{\vee}\boxtimes std\quad\mathrm{and}\quad r_{2}=\mathrm{sim}^{-1}\otimes Sym^{2},\]
where $std$ is the standard representation and sim is the similitude character of $GSp_{4}(\mathbb{C}).$


If $\bar{P}$ is the opposite parabolic and $\bar{P}^{\vee}$ is the dual of the opposite parabolic, then on the opposite nilpotent radical $\bar{\mathfrak{n}}^{\vee},$ the adjoint action of $M^{\vee}$ is the dual representation $r_{1}^{\vee}\oplus r_{2}^{\vee}.$ Shahidi decomposed the Plancherel measure as a product of gamma factors. 
More precisely [Sh, Thm. 3.5]:
\begin{prop} Suppose that $\pi\boxtimes\sigma$ is a generic representation of $M(k)=GSp_{4}(k)\times GL_{2m}(k).$ Then 
\begin{displaymath} \mu(s, \pi\boxtimes\sigma)=\gamma(s,\pi\boxtimes\sigma, r_{1},\psi)\gamma(s,\pi\boxtimes\sigma, r_{1}^{\vee},\psi)\gamma(2s,\pi\boxtimes\sigma, r_{2},\psi)\gamma(2s,\pi\boxtimes\sigma, r_{2}^{\vee},\psi).\end{displaymath} \end{prop} 
The local factors satisfy
\begin{displaymath}\gamma(s,\pi\boxtimes\sigma,r_{i},\psi)= \epsilon(s,\pi\times\sigma,r_{i},\psi)\cdot \frac{L(1-s,(\pi\times\sigma)^{\vee},r_{i})}{L(s,\pi\times\sigma,r_{i})},\end{displaymath}
where the factors on the right hand side were defined by Shahidi to satisfy the given decomposition of the Plancherel measure.

\sloppy

Collecting [GT, Cor 9.3], [GT, Thm 9.4], and [GTan, \S 8] we have
\begin{prop} Let  $\pi$ be an irreducible supercuspidal representation of $GSpin_5(k)$ or $GSpin_{4,1}(k)$ with parameter $\phi_\pi$ given by the local Langlands conjectures for $GSp_4$ and its inner form. Then if $\sigma$ is any irreducible supercuspidal representation of $GL_{2m}(k)$ with parameter $\phi_\sigma$, the Plancherel measure $\mu(s,\pi\boxtimes\sigma)$ is equal to
\begin{align*} &\gamma(s,\phi_{\pi}\otimes \phi_{\sigma},r_{1},\psi)\gamma(s,\phi_{\pi}\otimes \phi_{\sigma},r_{1}^{\vee},\bar{\psi})\gamma(2s,\phi_{\pi}\otimes \phi_{\sigma},r_{2},\psi)\gamma(2s,\phi_{\pi}\otimes \phi_{\sigma},r_{2}^{\vee},\bar{\psi})= \\
 \gamma (s,&\phi_{\pi}^\vee\otimes \phi_{\sigma},\psi)\gamma(-s,\phi_{\pi}\otimes \phi_{\sigma}^\vee,\bar{\psi})\gamma(2s,Sym^2\phi_\sigma\otimes\mr{sim}\phi_\pi^{-1},\psi)\gamma(-2s,Sym^2\phi_\sigma^\vee\otimes\mr{sim}\phi_\pi,\bar{\psi}).\end{align*}
\end{prop}

\fussy

Here
\bdm \gamma(s,\phi_\pi\otimes\phi_\sigma,r_{i},\psi)= \epsilon(s,\phi_\pi\otimes\phi_\sigma,r_{i},\psi)\cdot \frac{L(1-s,(\phi_\pi\otimes\phi_\sigma)^{\vee},r_{i})}{L(s,\phi_\pi\otimes\phi_\sigma,r_{i})},\edm
are the local factors of Artin type associated to the given representations of the Weil-Deligne group $W'_k.$
For a representation $\phi_1\otimes\phi_2$ of $W'_k$ the Artin L-function $L(s,\pi_1\otimes\pi_2)$ is given by
\bdm L(s,\phi_1\otimes\phi_2)=\frac{1}{\mr{det}(I-q^{-s}(\phi_1\otimes\phi_2)(\mr{Frob})\vert_{(V_{\phi_1}\otimes V_{\phi_2})^\mc{I}})}.\edm

\begin{lemma} Let  $\pi$ be an irreducible supercuspidal representation of $GSpin_5(k)$ or $GSpin_{4,1}(k)$ with $L$-parameter $\phi_\pi$ given by the local Langlands conjectures for $GSp_4$ and its inner form. 
Let $\sigma$ be an irreducible supercuspidal representation of $GL_{2m}(k)$ such that its  $L$-parameter $\phi_\sigma$ factors through $GSp_{2m}(\mb{C})$ with similitude character 
$\mr{sim}\phi_\pi.$ Then 
\bdm \mu(0,\pi\boxtimes\sigma)=0\quad \Longrightarrow\quad \mathrm{Hom}_{W_k}(\phi_\pi,\phi_\sigma)\ne 0.\edm \end{lemma}

\begin{proof} Let $\pi$ be a representation of $GSpin_5(k)$ or $GSpin_{4,1}(k)$ with parameter $\phi_\pi$ and $\sigma$ a representation of $GL_{2m}(k)$ with parameter $\phi_\sigma$ as in the statement of the lemma. By Proposition 8.2 we have 
\begin{align*} \mu(s,\pi\otimes\sigma)= & \gamma(s,\phi_{\pi}^\vee\otimes \phi_{\sigma},\psi)\gamma(-s,\phi_{\pi}\otimes \phi_{\sigma}^\vee,\bar{\psi})\\
 & \gamma(2s,Sym^2\phi_\sigma\otimes\mr{sim}\phi_\pi^{-1},\psi)\gamma(-2s,Sym^2\phi_\sigma^\vee\otimes\mr{sim}\phi_\pi,\bar{\psi})\\ = &
\epsilon-\mathrm{factors}\cdot
\frac{L(1-s,[\phi_\pi^{\vee}\otimes\phi_{\sigma}]^{\vee})}{L(s,\phi_\pi^{\vee}\otimes\phi_{\sigma})}\cdot
\frac{L(1+s,[\phi_\pi\otimes\phi_{\sigma}^{\vee}]^{\vee})}{L(-s,\phi_\pi\otimes\phi_{\sigma}^{\vee})} \\
& \cdot
\frac{L(1-2s,[Sym^2\phi_\sigma\otimes\mathrm{sim}\phi_\pi^{-1}]^{\vee})}{L(2s,Sym^2\phi_\sigma\otimes\mathrm{sim}\phi_\pi^{-1})}\cdot
\frac{L(1+2s,[Sym^2\phi_\sigma^\vee\otimes\mathrm{sim}\phi_\pi^{-1}]^{\vee})}{L(-2s,Sym^2\phi_\sigma^\vee\otimes\mathrm{sim}\phi_\pi)}.\end{align*}
Let $\mu(0,\pi\boxtimes\sigma)=0.$ 
From the expression for Artin L-functions given above, we can see that none of the numerators has a zero at $s=0.$
We have that $\phi_\sigma$ is irreducible and symplectic with similitude character $\mr{sim}\phi_\pi.$ By Schurs lemma it cannot also be orthogonal with similitude character $\mr{sim}\phi_\pi.$
Therefore
neither $\mr{Sym}^2\phi_\sigma\otimes\mr{sim}\phi_\pi^{-1}$ or $\mr{Sym}^2\phi_\sigma^\vee\otimes\mr{sim}\phi_\pi$ can contain a nonzero fixed vector under $W_k$ and neither of the last two denominators has a pole.  
This forces one of the first two denominators to have a pole. 
Therefore $\phi_\pi^\vee\otimes\phi_\sigma$ or $\phi_\pi\otimes\phi_\sigma^\vee$ contains the trivial representation and 
\bdm\mr{Hom}_{W_k}(\phi_\pi,\phi_\sigma)\ne 0.\edm
\end{proof}
  
We can now prove:

\begin{theorem}
Let $\phi$ be a tame regular discrete series $L$-parameter.
Let $L_\phi^{DR}$ be the $L$-packet of depth zero supercuspidal representations of $GSp_4(k)$ or $GU_2(D)$ corresponding to $\phi$ by the construction of DeBacker and Reeder given in Section 4.5. Let $L_\phi^{GT}$ be the $L$-packet of supercuspidal representations of $GSp_4(k)$ or $GU_2(D)$ corresponding to $\phi$ via the local Langlands conjecture for $GSp_4$ or $GU_2(D).$ Then
\bdm L_\phi^{DR}=L_\phi^{GT}.\edm  \end{theorem}

\begin{proof} In the following assume that all representations $\pi$ are unitary.
Let $\phi=\phi_{1}\oplus\dots\oplus\phi_{r}$, $r=1,2,$ be a tame regular discrete series Langlands parameter. Let $\pi_\phi$ be a representation of $GSp_4(k)$ or $GU_2(D)$ corresponding to $\phi$ as in Section 4.5. 
Under the correspondence defined by Gan and Takeda for $GSp_4$, or Gan and Tantono for $GU_2(D),$ $\pi_{\phi}$ corresponds to some $L$-parameter we call $\phi'.$
Let $\sigma=\sigma_{\phi_i}$, $1\le i\le r$, be the depth zero supercuspidal representation of $GL_{2m}$ attached to $\phi_i$ via the local Langlands correspondence for $GL_{2m}.$
Note that if $r=1$ then $m=2$, and if $r=2$ then $m=1.$

By Theorem 7.4, $I(s,\pi_{\phi}\boxtimes\sigma)$ reduces for some $s_{0}>0.$ By Proposition 7.2 this implies $\mu(0,\pi_{\phi}\boxtimes\sigma)=0.$ Then, by Lemma 8.3
\bdm\mathrm{Hom}_{W_k}(\phi_{i},\phi')\ne 0.\edm This holds for $1\le i\le r .$ Since for $r=2,$ $\phi_{1}\ncong\phi_{2},$ 
\begin{displaymath} \phi'=\phi\end{displaymath} in all cases. Therefore $L_\phi^{DR}=L_\phi^{GT}.$ 
\end{proof}
 
\begin{corollary} The parametrization of DeBacker and Reeder of depth zero supercuspidal representations of $GSp_{4}(k)$ arising from tame regular discrete series Langlands parameters coincides with the parametrization of Gan and Takeda. \end{corollary}

\begin{proof} Let $\phi$ be a tame regular discrete series parameter for $GSp_4(k).$ By Lemma 4.3 and Theorem 8.4,  the $L$-packet $L_\phi^{DR}$ of representations attached to $\phi$ by DeBacker and Reeder agrees with the $L$-packet $L_\phi^{GT}$ given by the local Langlands conjecture for $GSp_4.$
For $L$-packets of size two, by [GT, Main Thm (ii)] $L_\phi^{GT}$ contains exactly one generic representation indexed by the trivial character of $A_\phi.$
By [DR, 6.2.1] the generic representation in $L_\phi^{DR}$ is also indexed by the trivial character of $A_\phi$. Therefore the parametrizations agree.
\end{proof}

\appendix

\section{Root datum}

Here we give a description of the root datum for the various groups appearing in the paper.
Let $G$ be a connected reductive linear algebraic group and $T$ a maximal torus of $G.$ Let $X=X^*(T)$ and  $X^\vee=X_*(T)$ be the groups of algebraic characters and cocharacters of $T$. Let $\Phi$ and  $\Phi^\vee$ be the sets of roots and coroots of $T$. 
The quadruple \bdm \Psi=(X,\Phi,X^\vee,\Phi^\vee)\edm is the root datum for $G$.
In the following instead of listing the roots $\Phi$ and coroots $\Phi^\vee$, we give $\Delta$ a set of simple roots for $T$ that generate $\Phi$, and $\Delta^\vee$ a set of simple coroots for $T$ that generate $\Phi^\vee.$

The root datum for $GL_n$ can be given by
\bdm X=\mb{Z}e_0\oplus\mb{Z}e_1\oplus\dots\oplus\mb{Z}e_n, \quad
X^\vee=\mb{Z}e_0^*\oplus\mb{Z}e_1^*\oplus\dots\oplus\mb{Z}e_n^*,\edm
\bdm \Delta=\{a_1=e_1-e_2, a_2=e_2-e_3,\dots, a_{n-1}=e_{n-1}-e_n\},\edm \bdm
\Delta^\vee=\{ a_1^\vee=e_1^*-e_2^*, a_2^\vee=e_2^*-e_3^*,\dots, a_{n-1}^\vee=e_{n-1}^*-e_n^*\}.\edm
The root datum for $GSpin_{2n+1}$ can be given by [AS, Prop 2.1]
\bdm X=\mb{Z}e_0\oplus\mb{Z}e_1\oplus\dots\oplus\mb{Z}e_n, \quad
X^\vee=\mb{Z}e_0^*\oplus\mb{Z}e_1^*\oplus\dots\oplus\mb{Z}e_n^*,\edm 
\bdm \Delta=\{a_1=e_1-e_2, a_2=e_2-e_3,\dots, a_{n-1}=e_{n-1}-e_n, a_n=e_n\},\edm
\bdm \Delta^\vee=\{a_1^\vee=e_1^*-e_2^*,a_2^\vee=e_2^*-e_3^*,\dots,a^\vee_{n-1}=e_{n-1}^*-e_n^*,a_n^\vee=2e_n^*-e_0^*\}.\edm
The root datum for $GSpin_{2n}$ can be given by [AS, Prop 2.1]
\bdm X=\mb{Z}e_0\oplus\mb{Z}e_1\oplus\dots\oplus\mb{Z}e_n, \quad
X^\vee=\mb{Z}e_0^*\oplus\mb{Z}e_1^*\oplus\dots\oplus\mb{Z}e_n^*,\edm 
\bdm \Delta=\{a_1=e_1-e_2, a_2=e_2-e_3,\dots, a_{n-1}=e_{n-1}-e_n, a_n=e_{n-1}+e_n\},\edm
\bdm \Delta^\vee=\{a_1^\vee=e_1^*-e_2^*, a_2^\vee=e_2^*-e_3^*,\dots, a^\vee_{n-1}=e_{n-1}^*-e_n^*, a_n^\vee=e_{n-1}^*+e_n^*-e_0^*\}.\edm
The root datum for $GSp_{2n}$ can be given by 
\bdm X=\mb{Z}e_0\oplus\mb{Z}e_1\oplus\dots\oplus\mb{Z}e_n, \quad
X^\vee=\mb{Z}e_0^*\oplus\mb{Z}e_1^*\oplus\dots\oplus\mb{Z}e_n^*,\edm 
\bdm \Delta=\{a_1=e_1-e_2,a_2=e_2-e_3,\dots, a_{n-1}=e_{n-1}-e_n, a_n=2e_n-e_0\},\edm
\bdm \Delta^\vee=\{a_1^\vee=e_1^*-e_2^*, a_2^\vee=e_2^*-e_3^*,\dots, a_{n-1}^\vee=e_{n-1}^*-e_n^*, a_n^\vee=e_n^*\},\edm
The root datum for $GSO_{2n}$ can be given by 
\bdm X=\mb{Z}e_0\oplus\mb{Z}e_1\oplus\dots\oplus\mb{Z}e_n, \quad
X^\vee=\mb{Z}e_0^*\oplus\mb{Z}e_1^*\oplus\dots\oplus\mb{Z}e_n^*,\edm 
\bdm \Delta=\{a_1=e_1-e_2, a_2=e_2-e_3,\dots, a_{n-1}=e_{n-1}-e_n, a_n=e_{n-1}+e_n-e_0\},\edm
\bdm \Delta^\vee=\{ a_1^\vee=e_1^*-e_2^*, a_2^\vee=e_2^*-e_3^*,\dots, a^\vee_{n-1}=e_{n-1}^*-e_n^*, a_n^\vee=e_{n-1}^*+e_n^*\}.\edm


Given a quadratic space $V$, if one has the decomposition $V=V_1\oplus V_2,$ with $V_i$ nondegenerate quadratic subspaces, then $SO(V_1)\times SO(V_2)\subset SO(V).$
If we restrict the covering 
\bdm 1\longrightarrow Z^0 \longrightarrow GSpin(V) \longrightarrow SO(V) \longrightarrow 1\edm to the subgroup $SO(V_1)\times SO(V_2)$ we get
\bdm  1\longrightarrow Z^0 \longrightarrow GSpin(V_1)\times GSpin(V_2)/\Delta GL_1 \longrightarrow SO(V_1)\times SO(V_2) \longrightarrow 1.\edm
Precisely, let \bdm (GSpin_{2m}\times GSpin_{2n+1})/\Delta GL_1=(GSpin_{2m}\times GSpin_{2n+1})/\{h_0^*(\lambda)g_0^*(\lambda):\,\lambda\in GL_1\}\edm where
$h_0^*$ and $g_0^*$ are given in the following lemma.
Let \bdm (GSO_{2m}\times GSp_{2n})^\circ=\{(g_1,g_2)\in GSO_{2m}\times GSp_{2n}:\, \mr{sim}(g_1)=\mr{sim}(g_2)\}.\edm
\begin{lemma} The root datum for $(GSpin_{2m}\times GSpin_{2n+1})/\Delta GL_1$ is given by 
\bdm X=\mb{Z}e_0\oplus\dots\oplus\mb{Z}e_{m+n},\quad X^\vee=\mb{Z}e_0^*\oplus\dots\oplus\mb{Z}e_{m+n}^*,\edm
\bdm \Delta=\{e_1-e_2,\dots, e_{m-1}-e_n, e_{m-1}+e_m\}\cup\{e_{m+1}-e_{m+2},\dots,e_{m+n-1}-e_{m+n},e_{m+n}\},\edm
\begin{align*} \Delta^\vee= & \{e_1^*-e_2^*,\dots, e_{m-1}^*-e_m^*, e_{m-1}^*+e_m^*-e_0^*\}\\ & \cup\{e_{m+1}^*-e_{m+2}^*,\dots, e_{m+n-1}^*-e_{m+n}^*, 2e_{m+n}^*-e_0^*\}.\end{align*}
Also, $(GSO_{2m}\times GSp_{2n})^\circ=((GSpin_{2m}\times GSpin_{2n+1})/\Delta GL_1)^\vee.$
\end{lemma}

\begin{proof}
We work with the root datum for $GSpin_{2m}$ and $GSpin_{2n+1}$ given above using the letter $h$ for $GSpin_{2m}$ and $g$ for $GSpin_{2n+1}.$
The character lattice for $GSpin_{2m}\times GSpin_{2n+1}$ is the $\mb{Z}$-span of
\bdm h_0, h_1,\dots, h_m, g_0, g_1, \dots, g_n.\edm
The characters for $G=(GSpin_{2m}\times GSpin_{2n+1})/\Delta GL_1$ are those which are trivial on 
\bdm \{h_0^*(\lambda)g_0^*(\lambda):\,\lambda\in GL_1\}.\edm The character lattice for $G$ is the 
$\mb{Z}$-span of \bdm h_0-g_0,h_1,\dots,h_m,g_1,\dots,g_n.\edm 
Using the $\mb{Z}$ pairing of the root datum, the cocharacter lattice is the $\mb{Z}$-span of 
\bdm \overline{h_0^*}=\overline{g_0^*}, \overline{h_1^*},\dots,\overline{h_m^*},\overline{g_1^*},\dots,\overline{g_n^*}.\edm
Set \bdm e_0=h_0-g_0,e_1=h_1,\dots, e_m=h_m,e_{m+1}=g_1,e_{m+n}=g_n\edm and 
\bdm e_0^*=\overline{g_0^*}, e_1^*=\overline{h_1^*},\dots,e_m^*=\overline{h_m^*},e_{m+1}^*=\overline{g_1^*},\dots,e_{m+n}^*=\overline{g_n^*}.\edm
Using this notation we see that the roots and coroots for $G$ are those given in the statement of the lemma.

Similarly, we work with the root datum for $GSO_{2m}$ given above using the letter $h$ and the root datum for $GSp_{2n}$ given above using the letter $g.$
The characters for $G'=(GSO_{2m}\times GSp_{2n})^\circ$ are equivalence classes of characters for $GSO_{2m}\times GSp_{2n}$. Two characters are equivalent if they have the same value on all elements of $G'.$
The character lattice for $G'$ is the $\mb{Z}$-span of \bdm \overline{1/2(h_0+g_0)}=\overline{h_0}=\overline{g_0}, \overline{h_1},\dots,\overline{h_m},\overline{g_1},\dots,\overline{g_n}.\edm
Using the $\mb{Z}$ pairing of the root datum, the cocharacter lattice is the $\mb{Z}$-span of 
\bdm h_0^*+g_0^*,h_1^*,\dots,h_m^*,g_1^*,\dots,g_n^*.\edm
Setting
\bdm e_0=\overline{1/2(h_0+g_0)}, e_1=\overline{h_1},\dots, e_m=\overline{h_m},e_{m+1}=\overline{g_1},\dots,e_{m+n}=\overline{g_n}\edm
and \bdm e_0^*=h_0^*+g_0^*, e_1^*=h_1^*,\dots, e_{m}^*=h_m^*, e_{m+1}^*=g_1^*,\dots, e+{m+n}^*=g_n^*,\edm
we see that the roots of $G'$ are the coroots of $G$, and the coroots of $G'$ are the roots of $G$.
\end{proof}

The center of $GSpin_{2n}$ is not connected. 
Let \bdm GSpin_{2n}^\sim=(GL_1\times GSpin_{2n})/\{(1,1),(-1,\zeta_0)\},\edm
where $\zeta_0=e_1^*(-1)e_2^*(-1)\dots e_n^*(-1)$ is an element in the center of $GSpin_{2n}.$
The root datum for $GSpin_{2n}^\sim$ can be given by [AS, 2.6]
\bdm X=\mb{Z}E_{-1}\oplus\mb{Z}E_0\oplus\dots\oplus\mb{Z}E_n,\quad X^\vee=\mb{Z}E_{-1}^*\oplus\mb{Z}E_0^*\oplus\dots\oplus\mb{Z}E_n^*,\edm
\bdm \Delta=\{E_1-E_2,\dots, E_{n-1}-E_n, E_{n-1}+E_n-E_{-1}\},\edm
\bdm \Delta^\vee=\{E_1^*-E_2^*,\dots, E_{n-1}^*-E_n^*, E_{n-1}^*+E_n^*-E_0^*\}.\edm
The center of $GSpin_{2n}^\sim$ is the set of elements which that belong to the kernel of all the simple roots, namely
\bdm \{E_0^*(\mu)E_1^*(\nu)\dots E_{n}^*(\nu)E_{-1}^*(\nu^2):\,\,\mu,\nu\in GL_1\}\simeq GL_1\times GL_1, \edm which is connected.

Let \bdm (GSpin^\sim_{2m}\times GSpin_{2n+1})/\Delta GL_1=(GSpin^\sim_{2m}\times GSpin_{2n+1})/\{E_0^*(\lambda)e_0^*(\lambda):\,\lambda\in GL_1\}.\edm
\begin{lemma} The root datum for $(GSpin^\sim_{2m}\times GSpin_{2n+1})/\Delta GL_1$ is given by 
\[ X=\mb{Z}E_{-1}\oplus\mb{Z}E_0\oplus\dots\oplus\mb{Z}E_{m+n},\quad X^\vee=\mb{Z}E_{-1}^*\oplus\mb{Z}E_0^*\oplus\dots\oplus\mb{Z}E_{m+n}^*,\]
\begin{align*} \Delta= & \{E_1-E_2,\dots, E_{m-1}-E_n, E_{m-1}+E_m-E_{-1}\}\\ & \cup\{E_{m+1}-E_{m+2},\dots,E_{m+n-1}-E_{m+n},E_{m+n}\},\\ 
\Delta^\vee= & \{E_1^*-E_2^*,\dots, E_{m-1}^*-E_m^*, E_{m-1}^*+E_m^*-E_0^*\}\\ & \cup\{E_{m+1}^*-E_{m+2}^*,\dots, E_{m+n-1}^*-E_{m+n}^*, 2E_{m+n}^*-E_0^*\}.\end{align*}
\end{lemma}

\begin{proof} The proof is similar to that of Lemma A.1. \end{proof}

The center of $(GSpin_{2m}^\sim\times GSpin_{2n+1})/\Delta GL_1$ is given by
\bdm \{E_0^*(\mu)E_1^*(\nu)\dots E_{m}^*(\nu)E_{-1}^*(\nu^2):\,\,\mu,\nu\in GL_1\}\simeq GL_1\times GL_1, \edm which is connected.

\end{document}